\documentclass[10pt,a4paper]{article}

\usepackage{a4wide}

\usepackage{etex}

\usepackage{tikz} 

\newcounter{cA}\newcounter{cB}
\newlength{\lA}\newlength{\lB}\newlength{\lC}
\usetikzlibrary{positioning}
\usepackage{pgfplots}

\usepackage{mathtools}

\usepackage{IEEEtrantools}
\IEEEeqnarraydefcol{eqnarrayCscript}{\hfil$\scriptstyle}{$\hfil}

\usepackage{authblk}
\usepackage{amsmath,amsthm,amssymb}
\usepackage{enumerate}

\newtheorem{theorem}{Theorem}
\newtheorem{lemma}{Lemma}
\newtheorem{corollary}{Corollary}
\newtheorem{proposition}{Proposition}
\newtheorem{problem}{Problem}
\newtheorem{assumption}{Assumption}

 \newcommand{\II}[1]{\ensuremath{\mathbf{1}_{#1}}}
 \newcommand{\Iset}[2]{\ensuremath{\mathbf{1}_{#2}(#1)}}

\newcommand{\Halmos}{}

\usepackage[colorlinks=true,breaklinks=true,bookmarks=true,urlcolor=blue,
     citecolor=blue,linkcolor=blue,bookmarksopen=false,draft=false]{hyperref}

\newcommand{\fm}[1]{\scriptsize\mbox{\ensuremath{#1}}}

\newcommand{\ie}{{i.e.},\ }
\newcommand{\eg}{{e.g.},\ }
\newcommand{\cf}{{cf.} }

\newcommand{\NN}{\mathbb{N}}

\newcommand{\RR}{\mathbb{R}}

\renewcommand{\SS}{\mathcal{S}}

\newcommand{\FF}{\mathcal{F}}

\renewcommand{\k}{k}
\renewcommand{\t}{t}

\newcommand{\B}{N}
\newcommand{\N}{N}
\renewcommand{\S}{S}
\newcommand{\C}{C}
\newcommand{\Cset}{\mathcal{C}}
\renewcommand{\SS}{S}
\newcommand{\T}{T}

\renewcommand{\o}{0}

\newcommand{\mei}{\text{-}e_i}
\newcommand{\mea}{\text{-}e_1}
\newcommand{\meb}{\text{-}e_2}
\newcommand{\mda}{\text{-}d_1}
\newcommand{\mdb}{\text{-}d_2}

\newcommand{\na}{n_1}
\newcommand{\nb}{n_2}
\renewcommand{\ni}{n_i}
\newcommand{\ra}{r_1}
\newcommand{\rb}{r_2}
\newcommand{\ri}{r_i}

\newcommand{\ppp}{p_{4,d_1}}
\newcommand{\ppo}{p_{4,e_1}}
\newcommand{\ppm}{p_{4,d_2}}
\newcommand{\pmp}{p_{4,\text{-}d_2}}
\newcommand{\pmo}{p_{4,\text{-}e_1}}
\newcommand{\pmm}{p_{4,\text{-}d_1}}
\newcommand{\pop}{p_{4,e_2}}
\newcommand{\poo}{p_{4,\o}}
\newcommand{\pom}{p_{4,\text{-}e_2}}
\newcommand{\hpp}{p_{1,d_1}}
\newcommand{\hpo}{p_{1,e_1}}
\newcommand{\hmp}{p_{1,\text{-}d_2}}
\newcommand{\hmo}{p_{1,\text{-}e_1}}
\newcommand{\hop}{p_{1,e_2}}
\newcommand{\hoo}{p_{1,\o}}
\newcommand{\vpp}{p_{2,d_1}}
\newcommand{\vpo}{p_{2,e_1}}
\newcommand{\vpm}{p_{2,d_2}}
\newcommand{\vop}{p_{2,e_2}}
\newcommand{\voo}{p_{2,\o}}
\newcommand{\vom}{p_{2,\text{-}e_2}}
\newcommand{\rpp}{p_{3,d_1}}
\newcommand{\rpo}{p_{3,e_1}}
\newcommand{\rop}{p_{3,e_2}}
\newcommand{\roo}{p_{3,\o}}

\begin{document}

\title{A Linear Programming Approach to Error Bounds for Random Walks in the Quarter-plane}

\author[1,2]{Jasper Goseling}
\author[1]{Richard J. Boucherie}
\author[1]{Jan-Kees van Ommeren}
\affil[1]{Stochastic Operations Research, University of Twente, The Netherlands}
\affil[2]{Department of Intelligent Systems, Delft University of Technology, The Netherlands}
\affil[ ]{\small\texttt{\{j.goseling,r.j.boucherie,j.c.w.vanommeren\}@utwente.nl}}

\maketitle

\begin{abstract}
{We consider the approximation of the performance of random walks in the quarter-plane. The approximation is in terms of a random walk with a product-form stationary distribution, which is obtained by perturbing the transition probabilities along the boundaries of the state space. A Markov reward approach is used to bound the approximation error. The main contribution of the work is the formulation of a linear program that provides the approximation error.}
\end{abstract}

%
%
%
\section{Introduction} \label{sec:intro}
We consider random walks in the quarter-plane, \ie discrete-time Markov processes on state space $\S=\{0,1,\dots\}^2$. The random walks are homogeneous in the sense that within the interior of the state space, $\{1,2,\dots\}^2$, the transition probabilities are translation invariant. In both axes and in the origin of the state space --- \ie in $\{1,2,\dots\}\times\{0\}$, $\{0\}\times\{1,2,\dots\}$ and $\{(0,0)\}$ --- the transition probabilities are possibly distinct, but again translation invariant. Our interest is in steady-state behavior. More precisely, for a random walk $R$ with stationary distribution $\pi:\S\to[0,\infty)$, our interest is in
 $\FF = \sum_{n\in\S}F(n)\pi(n)$,
for some performance measure $F:\S\to[0,\infty)$. In particular, our interest is in characterizing the performance of the random walk by finding upper and lower bounds on $\FF$. 

Our approach to bounding the performance is based on two observations. The first observation is that closed form results for $\FF$ are readily obtained for the case that the stationary distribution $\pi$ is known to have a geometric product form. The second observation is that by carefully perturbing the transition probabilities of $R$ one obtains a random walk $\bar R$ for which the stationary distribution $\bar\pi$ has a geometric product form. Hence, the performance $\bar\FF$ of $\bar R$ is known in closed form. The basic idea of our approach is to bound the performance of $R$ in terms of $\bar\FF$. The main contribution of the current work is to show that $|\FF - \bar\FF|$ can be bounded by the solution of a linear program. In particular, we construct such a linear program in which the transition probabilities of $R$ and $\bar R$, the stationary distribution $\bar \pi$, and the performance measure $F$ are the only input parameters. Hence, this linear program is universal, in the sense that it can be used to obtain a bound on $|\FF-\bar\FF|$ without any additional preprocessing.

The current work builds on the Markov reward approach to error bounds as introduced by van Dijk and Puterman~\cite{van1988simple}. The method has since been further developed by van Dijk~\cite{vandijk88perturb, vandijk1998anor} and has been applied to, for instance, Erlang loss networks~\cite{boucherie2009monotonicity}, to tandem networks with finite buffers~\cite{vandijk88tandem}, to networks with breakdowns~\cite{van1988simple}, to queueing networks with non-exponential service~\cite{vandijk2004nonexp} and to wireless communication networks with network coding~\cite{goseling13peva}. An extensive description and overview of various applications of this method can be found in~\cite{vandijk11inbook}. The error bounding method provides a framework for establishing bounds on $|\FF - \bar\FF|$. Starting from the observation that $\FF$ can be interpreted as the average reward over an infinite time horizon in a Markov reward process, van Dijk formulates a bound on $|\FF - \bar\FF|$ in terms of bounds of the \emph{bias terms} (a.k.a.\ relative gains) of this Markov reward process. In addition to bounding the bias terms, the method is based on allowing a different reward function $\bar F$ on the perturbed process.

A major disadvantage of the error bound method is that the verification steps that are required in application of the method can be technically quite complicated. Indeed, no generic verification procedure is available in the literature and existing results depend on case by case verification by means of cumbersome induction proofs. The \emph{main contribution} of the current work consists of developing such a verification technique for random walks in the quarter-plane. The verification technique is based on formulating the application of the error bound method as a linear program. In doing so, it avoids the induction proofs completely. Moreover, if error bounds exist, the optimization framework will inherently lead to the best possible error bounds. Finally, the method uses piecewise linear functions to obtain bounds. It will be illustrated that the error bounds that are obtained based on piecewise linear functions would most likely not have been found with the approaches to error bounds that have so far been used in the literature.   

Our method depends on perturbing some of transition rates in order to get a product-form stationary distribution. It was shown in~\cite{bayer2002structure} that for continuous-time Markov processes in the quarter plane, such perturbations can always be found. A related result was presented in~\cite{latouche2013level,kroese2004spectral} for a (discrete-time) QBD processes that satisfy a technical condition. In~\cite{chen14paper3} the existence of such perturbations is demonstrated for all random walks in the quarter-plane. In the current work our concern is not with constructing the perturbed process. We assume that two processes are given and establish a bound on the difference in performance.

Another means of establising a relation between $\FF$ and the performance of the perturbed random walk $\bar R$ is through stochastic comparison~\cite{muller2002comparison}. The advantage of the error bound method over stochastic comparison is that it not only provides a comparison result on two systems, but also quantifies the performance difference between the two. In addition, the error bound method is able to provide results in cases that stochastic comparison results do not exist, see, for instance,~\cite{taylor1998strong}. 

While it is possible to obtain closed form expressions for $\FF$ in special cases, \eg for random walks with a product-form stationary distribution, no methods exist that provide such results for arbitrary random walks. There are some methods to find expressions for the generating functions of $\pi$, \cf~\cite{fayolle, cohenboxma}. However, these expressions can, in general, not be used for a straightforward calculation of $\FF$. In addition, these methods can not be straightforwardly applied. More precisely, they require a careful analysis of the the model and an adjustment of the method based on, \eg the transition probabilities.

Linear programming has been introduced by Kumar and Kumar~\cite{kumar1994performance} for bounding the performance of multiclass queueing networks. The goal is to establish performance bounds that hold for any stable scheduling policy. The method, which was generalized by Bertsimas et al. in~\cite{bertsimas1994optimization} and by Morrison and Kumar in~\cite{morrison1999new} relies on approximating the underlying average-cost Markov decision process. It was shown by de Farias and Van Roy~\cite{de2003linear,de2006cost} how this method  fits into a general linear programming approach to approximate dynamic programming. Another means of approximating the behavior of a random walk is to analyze the tail asymptotics. An overview of such methods is given in~\cite{miyazawa2009tail}. The most important difference between~\cite{kumar1994performance, bertsimas1994optimization, morrison1999new,miyazawa2009tail} is that in the current work we provide a bound on the performance difference of two processes with fixed policies.

The remainder of this paper is organized as follows. In Section~\ref{sec:prelim} we provide an exact statement of our model and the problem formulation. In Section~\ref{sec:motivation} we provide an introduction to the Markov reward approach to error bounds as well as an example that motivates our goal of developing a linear programming framework for obtaining error bounds. The linear programming approach to the error bound method is developed in Section~\ref{sec:result} for the case that the transition probabilities of $R$ and $\bar R$ differ only for transitions along the unit directions. An extension of the method to the general case, as well as some variations of the method are presented in Section~\ref{sec:generalize}. Examples that illustrate application of the method are given in Section~\ref{sec:example}. Finally, in Section~\ref{sec:disc} we provide a discussion of the current work and an outlook on future work.

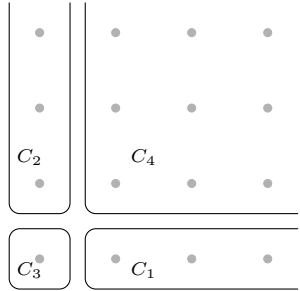
\begin{figure}
\centering
\begin{tikzpicture}[scale=1]
\setcounter{cA}{3}
\setcounter{cB}{3}
\setlength{\lA}{4mm}
\setlength{\lB}{3cm}
\setlength{\lC}{1.5mm}
\tikzstyle{vertex}=[draw,circle,fill,color=black!30!white,outer sep=0pt, inner sep=0pt, minimum size=3pt]
\tikzstyle{part}=[rounded corners]
\tikzstyle{partlabel}=[anchor=west,xshift=-\lA+1mm,yshift=-\lC,inner sep=0pt, outer sep=0pt]
\foreach \x in {0,...,\thecA}
 	\foreach \y in {0,...,\thecB}
 		\node[vertex] (x\x y\y) at (\x,\y) {}; 
\draw[part] (-\lA,-\lA) rectangle (\lA,\lA);
\draw[part] (\lA+\lB,-\lA) -- (-\lA+1cm,-\lA) -- (-\lA+1cm,\lA) -- (\lA+\lB,\lA);
\draw[part] (-\lA,\lA+\lB) -- (-\lA,-\lA+1cm) -- (\lA,-\lA+1cm) -- (\lA,\lA+\lB);
\draw[part] (\lA+\lB,-\lA+1cm) -- (-\lA+1cm,-\lA+1cm) -- (-\lA+1cm,\lA+\lB);
\node[partlabel] at (0,0) {\fm{\C_3}};
\node[partlabel] at (1.5,0) {\fm{\C_1}};
\node[partlabel] at (0,1.5) {\fm{\C_2}};
\node[partlabel] at (1.5,1.5) {\fm{\C_4}};
 \end{tikzpicture}

\caption{Partition of state space $\S$ into components $\C_1,\dots,\C_4$.}
 \label{fig:part}
\end{figure}

\begin{figure}
\centering
\beginpgfgraphicnamed{pgftrans}
\begin{tikzpicture}[scale=.8]
\tikzstyle{axes}=[very thin]
\tikzstyle{trans}=[very thick,-latex]
\tikzstyle{intloop}=[->,to path={
    .. controls +(30:3.5) and +(-30:3.5) .. (\tikztotarget) \tikztonodes}]
\tikzstyle{hloop}=[-latex,to path={
    .. controls +(-60:.6) and +(-120:.6) .. (\tikztotarget) \tikztonodes}]
\tikzstyle{vloop}=[->,to path={
    .. controls +(-150:.6) and +(-210:.6) .. (\tikztotarget) \tikztonodes}]
\tikzstyle{oloop}=[->,to path={
    .. controls +(255:.6) and +(195:.6) .. (\tikztotarget) \tikztonodes}]
\draw[axes] (0,0)  -- node[at end, below] {$\scriptstyle \rightarrow \na$} (8,0);
\draw[axes] (0,0) -- node[at end, left] (upa) {$\scriptstyle \nb$} (0,6.5);
\node[anchor=north,below=-1pt of upa,inner sep=0pt] {$\scriptstyle {\uparrow}$} ;
\draw[trans] (0,0) to node[below right,at end] {\fm{\rpo}} (1,0);
\draw[trans] (0,0) to node [at end, anchor=south west] {\fm{\rpp}} (1,1);
\draw[trans] (0,0) to node[above left, at end] {\fm{\rop}} (0,1);
\draw[trans] (5,0) to node[at end, below left] {\fm{\hmo}} (4,0);
\draw[trans] (5,0) to node[at end, below right] {\fm{\hpo}} (6,0);
\draw[trans] (5,0) to node[at end, anchor = south west] {\fm{\hpp}} (6,1);
\draw[trans] (5,0) to node[at end, anchor = south]  {\fm{\hop}} (5,1);
\draw[trans] (5,0) to node[at end, anchor = south east] {\fm{\hmp}} (4,1);
\draw[trans] (0,4) to node[below left, at end] {\fm{\vom}} (0,3);
\draw[trans] (0,4) to node[at end, anchor = west] {\fm{\vpo}} (1,4);
\draw[trans] (0,4) to node[above left, at end] {\fm{\vop}} (0,5);
\draw[trans] (0,4) to node[at end, anchor = north west] {\fm{\vpm}} (1,3);
\draw[trans] (0,4) to node[at end, anchor = south west] {\fm{\vpp}} (1,5);
\draw[trans] (5,4) to node[at end, anchor = west] {\fm{\ppo}} (6,4);
\draw[trans] (5,4) to node[at end, anchor = south west] {\fm{\ppp}} (6,5) ;
\draw[trans] (5,4) to node[at end, anchor = south] {\fm{\pop}} (5,5);
\draw[trans] (5,4) to node[at end, anchor = south east] {\fm{\pmp}} (4,5);
\draw[trans] (5,4) to node[at end, anchor = east] {\fm{\pmo}} (4,4);
\draw[trans] (5,4) to node[at end, anchor = north east] {\fm{\pmm}} (4,3);
\draw[trans] (5,4) to node[at end, anchor = north] {\fm{\pom}} (5,3);
\draw[trans] (5,4) to node[at end, anchor = north west] {\fm{\ppm}} (6,3);
\draw[very thick,-latex] (0,0) to[oloop] node[below] {\fm{\roo}} (0,0);
\draw[very thick,-latex] (5,0) to[hloop] node[below=.8mm] {\fm{\hoo}} (5,0);
\draw[very thick,-latex] (5,4) to[intloop] node[right] {\fm{\poo}} (5,4);
\draw[very thick,-latex] (0,4) to[vloop] node[left] {\fm{\voo}} (0,4);
\end{tikzpicture}
\endpgfgraphicnamed
\caption{Transition probabilities for random walk $R$.}
\label{fig:trans}
\end{figure}
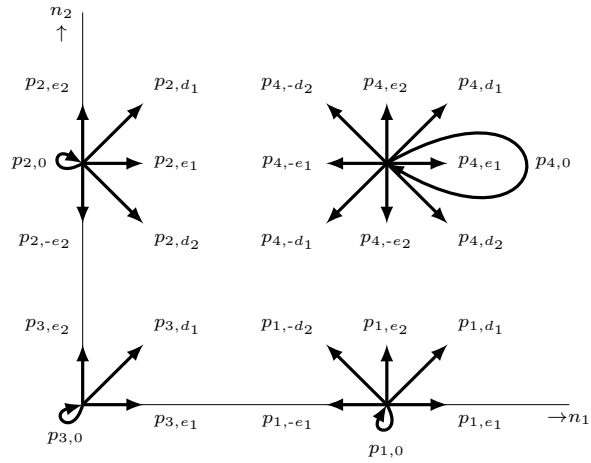

%
%
%
\section{Preliminaries} \label{sec:prelim}

\subsection{Model} \label{ssec:model}
We consider two random walks $R$ and $\bar R$, the state space of  which is the quarter plane, denoted by $\S$, \ie $\S = \{0,1,\dots\}\times\{0,1,\dots\}$. 
A state is represented by a pair of coordinates, \ie for $n\in\S$, $n=(\na,\nb)$.

We consider a partition of $\S$ into four components: $\C_1=\{1,2,\dots\}\times\{0\}$, $\C_2=\{0\}\times\{1,2,\dots\}$, $\C_3=\{(0,0)\}$ and $\C_4=\{1,2,\dots\}\times\{1,2,\dots\}$. We refer to these components as the horizontal axis, the vertical axis, the origin and the interior respectively. Let $k(n)$ denote the index of the component of state $n\in\S$, \ie $n\in\C_{\k(n)}$.

 We denote by $\B_k$ the neighbors of a state in $\C_k$. More precisely $\B_1=\{-1,0,1\}\times\{0,1\}$, $\B_2=\{0,1\}\times\{-1,0,1\}$, $\B_3=\{0,1\}\times\{0,1\}$ and $\B_4=\{-1,0,1\}\times\{-1,0,1\}$. Also, let $\B=\B_4$. For notational convenience we let $e_1=(1,0)$, $e_2=(0,1)$, $d_1=(1,1)$ and $d_2=(1,-1)$.

The random walks are discrete-time Markov processes, the transition probabilities of which are homogeneous in the sense that they are translation invariant in each of the components. Transitions are to neighbors only.  Let $p_{k,u}$ denote the probability of $R$ jumping from any state $n$ in component $\C_k$ to $n+u$, where $u\in\B_k$. Let $\bar p_{k,u}$ denote the corresponding probabillity for $\bar R$. For notational convenience let
\begin{equation}
q_{k,u} = \bar p_{k,u} - p_{k,u}.
\end{equation}
The partition into components and notation for transition probabilities are illustrated in Figures~\ref{fig:part} and~\ref{fig:trans}, respectively.

The stationary distributions of $R$ and $\bar R$, denoted by $\pi$ and $\bar\pi$, are the probability distributions that satisfy for all $n\in\S$,
\begin{equation*}
\pi(n) = \sum_{u\in N_{k(n)}}p_{k(n+u),\text{-}u}\pi(n+u) \quad\text{and}\quad \bar\pi(n) = \sum_{u\in N_{k(n)}}\bar p_{k(n+u),\text{-}u}\bar\pi(n+u),
\end{equation*}
respectively. We assume that $\bar\pi$ is a product-form geometric distribution, \ie that
\begin{equation} \label{eq:barpi}
\bar\pi(n) = \prod_{i=1,2}(1-\ri)\ri^{\ni}, 
\end{equation}
for some $r\in(0,1)\times(0,1)$ that is known. The stationary distribution $\pi$ is assumed to be unknown. 

\subsection{Problem statement} \label{ssec:problem}

Our goal is to establish upper and lower bounds on the steady-state performance of $R$ in terms of $\bar R$ and $\bar \pi$. The performance measure of interest is
\begin{equation}
\FF = \sum_{n\in\S} \pi(n)F(n),
\end{equation}
where $F:\S\to[0,\infty)$ is a function that is linear in each of the components of the state space, \ie
\begin{equation} \label{eq:clinear}
F(n) =
\begin{cases}
f_{1,0} + f_{1,1}\na, \quad &\text{if } n\in\C_1,\\
f_{2,0} + f_{2,2}\nb, \quad &\text{if } n\in\C_2,\\
f_{3,0}, \quad &\text{if } n\in\C_3,\\
f_{4,0} + f_{4,1}\na + f_{4,2}\nb, \quad &\text{if } n\in\C_4,
\end{cases}
\end{equation}
where $f_{k,i}$ are the constants that define the function. We refer to functions that are linear in each of the components of the state space as componentwise linear or as $\C$-linear. Let $\Cset$ denote the class of all $\C$-linear functions, $\Cset_+$ is the set of all non-negative $\C$-linear functions.

Finally, for $V\subset\B$ and $u\in\B$ let $V+u=\{w |\ w-u\in V\}$.

\subsection{Markov reward approach to error bounds} \label{ssec:errorbounds}
Our framework builds on the Markov reward approach for error bounds, an introduction to which is provided in~\cite{vandijk11inbook}. The gist of the approach is to interpret $f$ as a reward function, where $f(n)$ is the one-step reward if the random walk is in state $n$. We denote by $F^t(n)$ the expected cummulative reward at time $t$ if the random walk starts from state $n$ at time $0$, \ie
\begin{equation} \label{eq:dynF}
F^t(n) =
\begin{cases}
0,\quad &\text{if }t=0,\\
F(n) + \sum_{u\in \B_{k(n)}}p_{k(n),u}F^{t-1}(n+u),\quad &\text{if }t>0.
\end{cases}
\end{equation}
We will have particular interest terms of the form $D^t_u(n) = F^t(n+u) - F^t(n)$, which we refer to as {\em bias terms.} For the unit vectors, let $D_1^t(n)=D^t_{e_1}(n)$ and $D_2^t(n)=D^t_{e_2}(n)$.

The next results appears in, \eg~\cite{vandijk11inbook}, and provides a bound on the approximation error on $\FF$. We provide a presentation of the result for random walks in the quarter plane. A more general formulation of the result, applicable to arbitrary Markov chains, appears in~\cite{vandijk11inbook}. 
\begin{theorem}[\cite{vandijk11inbook}] \label{th:error}
Let $\bar F:\S\to [0,\infty)$ and $G:\S\to [0,\infty)$ satisfy
\begin{equation} \label{eq:constrtherror}
\Big| \bar F(n) - F(n) + \sum_{u\in\B_{k(n)}}q_{k(n),u}D_u^t(n) \Big| \leq G(n)
\end{equation}
for all $n\in\S$ and $t\geq 0$.
Then
\begin{equation*}
 \sum_{n\in\S}\left[\bar F(n) - G(n)\right]\bar \pi(n)\ \leq\ \FF\ \leq\ \sum_{n\in\S}\left[\bar F(n) + G(n)\right]\bar \pi(n).
\end{equation*}
\end{theorem}
The crucial element in the above theorem are the bias terms $D_u^t(n)$. It is in general not possible to find closed form expressions for the bias terms. Therefore, the usual means of applying the theorem is to find bounds on these bias terms. These bounds then lead to a function $G$ satisfying~\eqref{eq:constrtherror}. The difficulty in practice is that even finding suitable bounds on the bias terms is a challenging task.
The only means that is available in the literature for tightly bounding the bias terms is to carefully inspect the structure of the process at hand and meticulously craft suitable bounds. The main contribution of the current work is a means of establishing error bounds for random walk that do not require manual construction of bounds on the bias terms. 

We illustrate in the next section an application of Theorem~\ref{th:error} to an example. The purpose is to illustrate the difficulties mentioned above, but more importantly to introduce some of techniques that will be developed in Section~\ref{sec:result}.


\begin{figure}
\hfill
\begin{tikzpicture}[scale=.5]
\tikzstyle{axes}=[very thin]
\tikzstyle{trans}=[very thick,-latex]
\tikzstyle{intloop}=[->,to path={
    .. controls +(30:3) and +(-30:3) .. (\tikztotarget) \tikztonodes}]
\tikzstyle{hloop}=[-latex,to path={
    .. controls +(-60:1) and +(-120:1) .. (\tikztotarget) \tikztonodes}]
\tikzstyle{vloop}=[->,to path={
    .. controls +(-150:1) and +(-210:1) .. (\tikztotarget) \tikztonodes}]
\tikzstyle{oloop}=[->,to path={
    .. controls +(255:1) and +(195:1) .. (\tikztotarget) \tikztonodes}]

\draw[axes] (0,0)  -- node[at end, below] {$\scriptstyle \rightarrow n_1$} (6.5,0);
\draw[axes] (0,0) -- node[at end, left] {$\scriptstyle {\uparrow} {n_2}$} (0,6.5);
\draw[trans] (0,0) to node[below,at end] {\fm{\lambda_1}} +(1,0);
\draw[trans] (0,0) to node[left, at end] {\fm{\lambda_2}} +(0,1);
   
\draw[trans] (4,0) to node[at end, below] {\fm{\mu_1}} +(-1,0);
\draw[trans] (4,0) to node[at end, below] {\fm{\lambda_1}} +(1,0);
\draw[trans] (4,0) to node[at end, anchor = south]  {\fm{\lambda_2}} +(0,1);
   
\draw[trans] (0,4) to node[left, at end] {\fm{\mu_2}} +(0,-1);
\draw[trans] (0,4) to node[at end, anchor = west] {\fm{\lambda_1}} +(1,0);
\draw[trans] (0,4) to node[left, at end] {\fm{\lambda_2}} +(0,1);
    
\draw[trans] (4,4) to node[at end, anchor = west] {\fm{\lambda_1}} +(1,0);
\draw[trans] (4,4) to node[at end, anchor = south] {\fm{\lambda_2}} +(0,1);
\draw[trans] (4,4) to node[at end, anchor = north east] {\fm{\mu}} +(225:1);

\draw[very thick,-latex] (0,0) to[oloop] node[below] {\fm{\mu}} (0,0);
\draw[very thick,-latex] (4,0) to[hloop] node[below=.4mm] {\fm{\mu-\mu_1}} (4,0);
\draw[very thick,-latex] (0,4) to[vloop] node[left] {\fm{\mu-\mu_2}} (0,4);

\end{tikzpicture}
\hfill{}
\caption{
Random walk with joint departures.
\label{fig:nc}}
\end{figure}
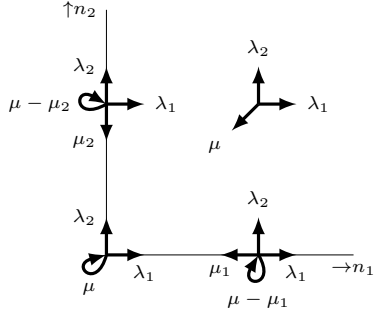

%
%
%
\section{Motivating example} \label{sec:motivation}
We consider a random walk arising from a queueing application in communication networks. The application is network coding in a two-way relay as recently studied in~\cite{goseling13peva}. For details on the application we refer the reader to~\cite{goseling13peva} and the references therein. The model corresponds to two queues with simultaneous departures from both queues. In case one of the queues is empty the other queue services packets at a lower rate. The non-zero transition probabilities in the corresponding random walk, obtained by uniformization of the continuous-time queueing model with Poisson arrivals and exponential service, are
\begin{equation}
p_{k,e_1} = \lambda_1,\quad p_{k,e_2}=\lambda_2,
\end{equation}
for $k=1,\dots,4$,
\begin{equation}
p_{1,\mea} = \mu_1,\quad p_{2,\meb} =\mu_2,\quad p_{4,\mdb}=\mu, 
\end{equation}
and
\begin{equation}
p_{1,\o} = \mu-\mu_1,\quad p_{2,\o}=\mu-\mu_2,\quad p_{3,0}=\mu,
\end{equation}
where $\lambda_1+\lambda_2+\mu=1$ and $\mu_i\leq \mu$, $i=1,2$. The normalization $\lambda_1+\lambda_2+\mu=1$ arises naturally from the uniformization of the continuous-time model and does not impose a restriction on the models that can be analyzed. The transition diagram is depicted in Figure~\ref{fig:nc}. We will refer to this process as the random walk with joint departures.

No closed form expression for the stationary distribution $\pi$ of this random walk is known in general.
Therefore, we consider the perturbed random walk $\bar R$, with
\begin{gather}
 \bar p_{1,\mea} = \bar\mu_1,\quad \bar p_{2,\meb} = \bar\mu_2,
\end{gather}
and $\bar p_{k,u} = p_{k,u}$ for other values of $k$ and $u$. In particular, we consider $\bar\mu_1+\bar\mu_2=\mu$, since in that case it is known~\cite{goseling13peva} that if $0<r_1<1$ and $0<r_2<1$ are the unique solution of
\begin{equation} \label{eq:prelimr}
\bar\mu_1 r_1+\bar\mu_2 r_1r_2 = \lambda_1,\quad
\bar\mu_2r_2+ \bar\mu_1 r_1r_2 = \lambda_2,
\end{equation}
then the stationary distribution of $\bar R$ is a geometric product-form,
 $\bar \pi(n) = (1-r_1)r_1^{\na}(1-r_2)r_2^{\nb}.$

The performance measure that we consider is the probability that both queues are empty, \ie we consider
\begin{equation}
F(n)=
\begin{cases}
1, \quad &\text{if }n=(0,0),\\
0, \quad &\text{otherwise}
\end{cases}
\end{equation}
and we are interested in $\FF=\sum_{n\in\SS}F(n)\pi(n)$. The reason that we consider this performance measure is that with the techniques that are used in this section we have been unsuccessful in establishing results for other performance meaures like the expected number of customers in the first queue. The difficulty in establishing results for other performance measures is an important motivation for the current paper.

The challenge is to apply Theorem~\ref{th:error} and obtain bounds on $\FF$ in terms of $\bar\pi$. As indicated in the discussion below the statement of Theorem~\ref{th:error} we need to establish bounds on the bias terms $D_u^t(n)$. A first inspection of $\bar R$ and $R$ reveils that $q_{k,u}=\bar p_{k,u} - p_{k,u}$ are zero if $u\not\in\{-e_1,-e_2,\o\}$. Therefore, we need to establish only bounds on $D_{\mea}^t(n)$, $D^t_{\meb}(n)$ and $D^t_{\o}(n)$ in order to find functions $\bar F$ and $G$ satisfying~\eqref{eq:constrtherror}. Since, furthermore, $D^t_{\o}(n)=0$ and $-D_{-e_i}^t(n)=-D^t_i(n-e_i)$, we will only consider $D^t_1(n)$ and $D^t_2(n)$. We provide in the next proposition an expression for the bias terms at time $t+1$ in terms of the bias terms at time $t$. The result has been obtained by a careful examination of the particular structure of the random walk with joint departures and the performance measure at hand. We will use this recursive result on the bias terms to derive upper and lower bounds on these bias terms.
\begin{proposition} \label{prop:dynD}
Let $R$ be a random walk with joint departures and $F(n)=\II{n=(0,0)}$. Then
\begin{equation} \label{eq:prelimd1}
D_1^{t+1}(n) =
\begin{cases}
\sum_{i=1}^2 \lambda_iD_1^t(n+e_i) + \mu_1 D_1^t(n-e_1) + (\mu-\mu_1)D_1^t(n) ,\quad &\text{if }n\in\C_1,\\
\sum_{i=1}^2 \lambda_iD_1^t(n+e_i) - (\mu-\mu_2) D_2^t(n-e_2),\quad &\text{if }n\in\C_2,\\ 
-1 + \sum_{i=1}^2 \lambda_iD_1^t(n+e_i) + (\mu-\mu_1)D_1^t(n),\quad &\text{if }n\in\C_3,\\ 
\sum_{i=1}^2 \lambda_iD_1^t(n+e_i) + \mu D_1^t(n-d_1),\quad &\text{if }n\in\C_4\\ 
\end{cases}
\end{equation}
and
\begin{equation} \label{eq:prelimd2}
D_2^{t+1}(n) =
\begin{cases}
\sum_{i=1}^2 \lambda_iD_2^t(n+e_i) - (\mu-\mu_1)D_1^t(n-e_1) ,\quad &\text{if }n\in\C_1,\\
\sum_{i=1}^2 \lambda_iD_2^t(n+e_i) + \mu_2 D_2^t(n-e_2) + (\mu-\mu_2) D_2^t(n),\quad &\text{if }n\in\C_2,\\ 
-1 + \sum_{i=1}^2 \lambda_iD_2^t(n+e_i) + (\mu-\mu_2)D_2^t(n),\quad &\text{if }n\in\C_3,\\ 
\sum_{i=1}^2 \lambda_iD_2^t(n+e_i) + \mu D_2^t(n-d_1),\quad &\text{if }n\in\C_4,\\ 
\end{cases}
\end{equation}
for all $n\in\SS$ and $t>0$.
\end{proposition}
\begin{proof}{Proof:}
We prove~\eqref{eq:prelimd1} for the case that $n\in\C_2$. The proofs for the other cases and for~\eqref{eq:prelimd2} follow in similar fashion. Note that for $n\in\C_2$, $n+e_1\in\C_4$. We have
\begin{align}
 D_1^{t+1}(n)
=&\ F(n+e_1)-F(n) + \sum_{u\in\B_4} p_{4,u}F^t(n+e_1+u) - \sum_{v\in\B_{2}} p_{2,v}F^t(n+v) \\
=&\ \sum_{i=1}^2 \lambda_i F^t(n+e_1+e_i) + \mu F^t(n+e_1-d_1) \notag\\
  &\ - \sum_{i=1}^2 \lambda_i F^t(n+e_i) - \mu_2 F^t(n-e_2) - (\mu-\mu_2) F^t(n) \\
=&\ \sum_{i=1}^2 \lambda_iD_1^t(n+e_i) + (\mu-\mu_2) F^t(n-e_2) -  (\mu-\mu_2) F^t(n-e_2+e_2) \\
=&\ \sum_{i=1}^2 \lambda_iD_1^t(n+e_i) - (\mu-\mu_2) D^t_2(n-e_2),
\end{align}
where the first equality follows from~\eqref{eq:dynF} and the second equality from $n\in\C_2$ and the structure of the random walk with joint departures. \Halmos
\end{proof}
 The general method as presented in Section~\ref{sec:result} is also based on first establishing such a recursive relation on the bias terms. It is a priori not clear that such a relation can always be found. One of the results presented in this paper is that for random walks this is indeed possible. Moreover, we provide a structured means of finding such relation. This leverages the need for manual derivations as performed in development of Proposition~\ref{prop:dynD}.

The next proposition provides the actual bounds on the bias terms. For clarity of exposition we consider the symmetrical case that $\lambda_1=\lambda_2=\lambda$ and $\mu_1=\mu_2=\mu^*$.
\begin{proposition} \label{prop:prelimbias}
Let $R$ be a random walk with joint departures and $F(n)=\II{n=(0,0)}$. If $\lambda_1=\lambda_2$ and $\mu_1=\mu_2=\mu^*$ then
\begin{equation} \label{eq:prelimab}
-\frac{1}{\mu^*}\leq D_i^t(n) \leq \frac{\mu-\mu^*}{\mu\mu^*}
\end{equation}
for $i\in\{1,2\}$, $n\in\SS$ and $t\geq 0$.
\end{proposition}
\begin{proof}{Proof:}
For $t=0$~\eqref{eq:prelimab} holds since $F^0(n)=0$ for all $n\in\SS$ and $u^*\leq\mu$. The proof now follows from a simple induction on $t$ by verifying all eight cases in~\eqref{eq:prelimd1} and~\eqref{eq:prelimd2}.\Halmos
\end{proof}
Recall that $\lambda_1+\lambda_2+\mu=1$. Therefore, even though the value of $\lambda$ influences the bounds that can be given on the bias terms, the presentation of the above result could be given in terms of $\mu$ and $\mu^*$ only. The difficulty in establishing the equivalent

The main result of this subsection is provided in the next proposition. It provides upper and lower bounds on the probability that the random walk with joint departures is in the origin.
\begin{proposition} \label{prop:prelimbounds}
Let $R$ and $\bar R$ be random walks with joint departures, $\lambda_1=\lambda_2=\lambda$. Let $R$ have $\mu_1=\mu_2=\mu^*$, where $\mu^*<\mu/2$. Let $\bar R$ have $\bar\mu_1=\bar\mu_2=\mu/2$. Finally, let $F(n)=\II{n=(0,0)}$. Then
\begin{equation}
 (1-r)^2 - g \leq \FF \leq (1-r)^2 + g,
\end{equation}
where
\begin{equation}
r = \frac{-1+\sqrt{1+8\lambda/\mu}}{2},\quad g = 2r(1-r)\frac{(\mu/2-\mu^*)(\mu-\mu^*)}{\mu\mu^*}.
\end{equation}
\end{proposition}
\begin{proof}{Proof:}
For this particular $R$ and $\bar R$ we have
\begin{equation} \label{eq:prelimq}
q_{k,u} =
\begin{cases}
\mu/2-\mu^*,\quad &\text{if }k=i, u=-e_i, i\in\{1,2\},\\
\mu^*-\mu/2,\quad &\text{if }k\in\{1,2\}, u=\o,\\
0,\quad &\text{otherwise}.\\
\end{cases}
\end{equation}
From~\eqref{eq:prelimq} and the discussion leading to Proposition~\ref{prop:dynD} it follows that if $\bar F$ and $G$ satisfy
\begin{equation} \label{eq:prelimconstrtherror}
\Big| \bar F(n) - F(n) + \sum_{\mathclap{i=1,2}}q_{k(n),\mei}D^t_i(n-e_i) \Big| \leq G(n), 
\end{equation}
then they satisfy~\eqref{eq:constrtherror}. Let $\bar F(n)=F(n)$ and
\begin{equation}
G(n)=
\begin{cases}
\displaystyle \frac{\mu/2-\mu^*}{\mu^*},\quad &\text{if } n\in\C_1 \text{ or } n\in\C_2, \\
0,\quad &\text{otherwise}.
\end{cases} 
\end{equation}
Using the fact that $\mu/2-\mu^*>0$ it is readily verified that these $\bar F$ and $G$ satisfy~\eqref{eq:prelimconstrtherror}. By observing that $r=r_1=r_2$ is the unique positive solution of~\eqref{eq:prelimr} and that
\begin{equation}
 \sum_{n\in\C_1} \bar\pi(n) = \sum_{n\in\C_2} \bar\pi(n) = r(1-r),
\end{equation}
the result follows from Theorem~\ref{th:error}.\Halmos
\end{proof}
For the case that $\mu^*>\mu/2$ a similar result can easily be obtained. The assymetrical case $\lambda_1\neq\lambda_2$ and/or $\mu_1\neq\mu_2$ is significantly more challenging in the sense that without the tools that are developed in Section~\ref{sec:result} of the current paper, generalizing Proposition~\ref{prop:prelimbias} is mostly a matter of guessing the correct form of the bounds and verifying valididity. One of the contributions of this paper is to generate the bounds on the bias terms and the functions $\bar F$ and $G$ by solving a linear program.

An added benefit of formulating the construction of error bounds in an optimization framework is that we can use as an objective the minimization of the upper bound on $\FF$. This will produce, within the class of functions $\bar F$ and $G$ that are under consideration, the tightest possible error bound. In this subsection we obtained constant bounds on the bias terms and piecewise constant functions $\bar F$ and $G$. A natural question is to ask whether better bounds could have been obtained by allowing, for instance, piecewise constant functions for the bounds on the bias terms. The answer is affirmative. The improved bounds will be presented in Section~\ref{sec:example}. In Section~\ref{sec:example} we will also give performance bounds for other performance measures, for instance, the marginal first moments.

%
%
%
\section{A linear programming approach to error bounds} \label{sec:result}
In this section we will present our approach to the error bound method. We develop a linear program that provides an upper bound to the performance approach to finding the approximation error. We restrict our attention to the case that $R$ and $\bar R$ differ only for transitions that are along the unit directions, \ie throughout this section we assume that
\begin{equation}
q_{k,u}=\bar p_{k,u} - p_{k,u}=0\quad\text{for}\quad u\neq\{e_1,e_2,-e_1,-e_2,\o\}.
\end{equation}
The reason for this restriction is that it significantly simplifies the presentation of the result. A generalization of the result to arbitrary $R$ and $\bar R$ is given in Section~\ref{sec:generalize}. In Section~\ref{sec:generalize} we also present the corresponding result that provides a lower bound on the performance. 

The outline of this section is as follows. In Subsection~\ref{ssec:optbound} we formulate a first minimization problem that provides an upper bound on $\FF$. This problem can not be solved efficiently, since it depends on the unkown bias terms. Therefore, we develop a framework for bounding the bias terms in Subsection~\ref{ssec:boundbiasterms}. The main result of this section, the error bound result itself, is given in Subsection~\ref{ssec:errorbound}. In Subsection~\ref{ssec:finite} it is shown that the corresponding optimization problem is linear with a finite number of variables and a finite number of constraints.

\subsection{An optimized error bound} \label{ssec:optbound}
To start, consider the following optimization problem.
\begin{problem} \label{pr:minstart}
\begin{align}
 \text{minimize}\ &\sum_{n\in\S}\left[\bar F(n) + G(n)\right]\bar \pi(n), \\
\text{subject to}\
  &\Big| \bar F(n) - F(n) + \sum_{\mathclap{i=1,2}}\left(q_{k(n),e_i}D^t_i(n) + q_{k(n),\mei}D^t_i(n-e_i) \right)\Big|  \leq G(n),\quad \text{for }n\in\S, t\geq 0, \label{eq:prminstartconstr} \\
  &\bar F(n)\geq 0, G(n)\geq 0, \quad \text{for }n\in\S,
\end{align}
\end{problem}
The variables in Problem~\ref{pr:minstart} are the functions $\bar F$ and $G$; the functions $F$, $\bar\pi$ and $D_u^t$ are parameters. Alternatively we can interpret Problem~\ref{pr:minstart} as an optimization over variables $\bar F(n)$ and $G(n)$, with two such variables for each $n\in\S$. This directly leads to a linear optimization problem. Indeed the objective function in Problem~\ref{pr:minstart} is linear and the modulus in constraint~\eqref{eq:prminstartconstr} induces two linear inequalities for each $n\in\S$ and $t\geq 0$. This linear program has a countably infinite number of variables and constraints. Our main result, to be presented later in this section, is a reduction of the above problem to a linear program with a finite number of  variables and constraints.

Before proceeding, we show that the optimal value of Problem~\ref{pr:minstart} provides an upper bound on $\FF$. From $D^t_{\o}(n)=0$ and $D^t_{\mei}(n)=-D^t_{e_i}(n-e_i)$ it follows directly that if $q_{k,u}=\bar p_{k,u} - p_{k,u}=0$ for $u\neq\{e_1,e_2,-e_1,-e_2,\o\}$ then~\eqref{eq:constrtherror} is equivalent to~\eqref{eq:prminstartconstr}. Therefore, it follows from Theorem~\ref{th:error} that the optimal value of Problem~\ref{pr:minstart} provides an upper bound on $\FF$. The problem of maximizing $\sum_{n\in\S}\left[\bar F(n) - G(n)\right]\bar \pi(n)$ subject to the same constraints leads to a lower bound on $\FF$. Since the optimization problems providing the upper and the lower bound are closely related, we illustrate the development of our main result by means of Problem~\ref{pr:minstart}. The corresponding result for the lower bound will be given at the end of the section.

The most important difficulty in handling Problem~\ref{pr:minstart} is that constraint~\eqref{eq:prminstartconstr} is expressed in terms of the bias terms, \ie the unknown functions $D_u^t(n)$. As a first step in developing our linear program we introduce pairs of functions $A_i:\S\to [0,\infty)$ and $B_i:\S\to [0,\infty)$, $i=1,2$. In the next subsection we will formulate a finite number of constraints on these functions that guarantee that
\begin{equation} \label{eq:biasboundsrequired}
 -A_i(n) \leq D_i^t(n) \leq B_i(n),
\end{equation}
for all $t\geq 0$, \ie these functions provide bounds on the bias terms uniformly over all $t\geq 0$. For the moment we assume that constraints providing~\eqref{eq:biasboundsrequired} can be constructed and replace occurences of $D^t_u(n)$ with its bounds $-A_i(n)$ and $B_i(n)$. The advantage of doing so is that the new problem does not involve the unkown terms $D_i^t(n)$. In addition it reduces countably many constraints (one constraint for each $t\geq 0$) to a single constraint. By replacing in Problem~\ref{pr:minstart} occurences of $D_i^t(n)$ with its bounds $-A_i(n)$ and $B_i(n)$ we make the constraints more stringent, \ie the optimal value of an optimization problem based on these bounds still provides an upper bound on $\FF$.

We are now ready to formulate an optimization problem in terms of the functions $A_i$ and $B_i$. For clarity of exposition, we do not replace $D_u^t$ with $A_i$ or $B_i$ directly, but instead make use of auxiliary functions $E_i:\S\to\RR$, $i=1,2$. Replacing in Problem~\ref{pr:minstart} $D_i^t(n)$ with its bound leads to the following optimization problem.
\begin{problem} \label{pr:stepa}
\begin{align}
 \text{minimize}\ &\sum_{n\in\S}\left[\bar F(n) + G(n)\right]\bar \pi(n), \\
\text{subject to}\
  &\Big| \bar F(n) - F(n) + \sum_{\mathclap{i=1,2}}\left(q_{k(n),e_i}E_i(n) + q_{k(n),\mei}E_i(n-e_i) \right)\Big|  \leq G(n),\quad \text{for }n\in\S,  \\
  & -A_i(n) \leq E_i(n) \leq B_i(n),\quad \text{for }n\in\S, i\in\{1,2\},\\
  & -A_i(n) \leq D_i^t(n) \leq B_i(n),\quad \text{for }n\in\S, i\in\{1,2\}, t\geq 0, \label{eq:stepabiasbounds} \\
  &\bar F(n)\geq 0, G(n)\geq 0, 
\quad \text{for }n\in\S,
\end{align}
\end{problem}
In the above problem $\bar F$, $G$, $A_i$, $B_i$ and $E_i$ are the variables.
Obviously we did not really solve any of the underlying problems by reworking Problem~\ref{pr:minstart} into Problem~\ref{pr:stepa}. It remains to replace~\eqref{eq:stepabiasbounds} with constraints that do not involve the bias terms $D_i^t(n)$ themselves. Therefore, the aim of the next subsection is to provide such bounds on the bias terms.

\subsection{Bounding the bias terms} \label{ssec:boundbiasterms}
The goal of this subsection is to obtain constraints on $A_i:\SS\to[0,\infty)$ and $B_i:\SS\to[0,\infty)$ that ensure~\eqref{eq:stepabiasbounds}. These constraints are developed in an inductive framework, \ie based on an induction in $t$. Therefore, the first goal of this subsection is to provide a generalization of Proposition~\ref{prop:dynD}, by expressing $D_i^{t+1}$ as a linear combination of $D_1^t$ and $D_2^t$. Next, we will use this relation to develop the desired constraints.

Our first contribution is to show that we can always express $D_i^{t+1}$ as a linear combination of $D_1^t$ and $D_2^t$. More precisely, we introduce the constants $c_{i,k,j,u}$, $i\in\{1,2\}$, $k\in\{1,\dots,4\}$, $j\in\{1,2\}$, $u\in N_{k}$, and provide a set of sufficient conditions under which these constants satisfy
\begin{align} \label{eq:dynDb}
D_i^{t+1}(n)
&= F(n+e_i)-F(n) + \sum_{j=1,2}\sum_{u\in\B_{\k(n)}}c_{i,k(n),j,u}D_j^t(n+u).
\end{align}
One can think of $c_{i,k,j,u}$ as the contribution of $D^t_j(n+u)$ to $D_i^{t+1}(n)$ if $n\in\N_k$. In addition to the sufficient conditions we prove that there always exist values for $c_{i,k,j,u}$ that satisfy these conditions. In particular, we show that there exist a `universal' set of constants that can be used, \ie constants that are given by a fixed function of the transition probabilities.

First, sufficient conditions for~\eqref{eq:dynDb} are given. The result is expressed using $k[i]$. For $i=1,2$ we define $k[i]$ as $k(n+e_i)$, for $n\in C_k$. Recall from Section~\ref{sec:prelim} that $\B_k+e_i=\{u |\ u-e_i\in\B_k\}$. We formulate our conditions in the following assumption.
\begin{assumption} \label{ass:a}
The constants $c_{i,k,j,u}$, $i,j=1,2$, $k=1,\dots,4$, $u\in\B_k$, satisfy
\begin{multline} \label{eq:constantsconditionsmain}
\Iset{w-e_1}{\B_{k}}c_{i,k,1,w-e_1} - \Iset{w}{\B_{k}}c_{i,k,1,w}  + \Iset{w-e_2}{\B_{k}}c_{i,k,2,w-e_2} \\ - \Iset{w}{\B_{k}}c_{i,k,2,w}
= \Iset{w-e_i}{\B_{k[i]}}p_{k[i],w-e_i} - \Iset{w}{\B_k}p_{k,w}
\end{multline}
for all $i\in\{1,2\}$, $k\in\{1,\dots,4\}$ and $w\in \B_{k}\cup (\B_{k[1]}+e_1)\cup(\B_{k[2]}+e_2)$. 
\end{assumption}
We will show below that we can always find coefficients $c_{i,k,j,u}$ that satisfy the above assumption. Therefore, we will assume in the remainder of this paper that coefficients that satisfy Assumption~\ref{ass:a} have been given. Before, proving that such coefficients exist we will first provide a technical result that motivates the conditions in Assumption~\ref{ass:a}. The reason is that these conditions provide sufficient conditions for~\eqref{eq:dynDb}, a result that we formulate more precisely here.
\begin{lemma} \label{lem:dynD}
If Assumption~\ref{ass:a} holds then
\begin{equation} \label{eq:dynD}
D_i^{t+1}(n) = F(n+e_i)-F(n) + \sum_{j=1,2}\sum_{u\in\B_{\k(n)}}c_{i,k(n),j,u}D_j^t(n+u),
\end{equation}
for $i=1,2$, $n\in\S$ and $t>0$.
\end{lemma}
\begin{proof}{Proof}
Consider arbitrary $i\in\{1,2\}$, $n\in\S$ and $t>0$. For notational convenience, let $k=k(n)$.

From~\eqref{eq:dynF} it follows directly that
\begin{equation}
 D_i^{t+1}(n) =\ F(n+e_i)-F(n) + \sum_{u\in\B_{k[i]}} p_{k[i],u}F^t(n+e_i+u) - \sum_{v\in\B_{k}} p_{k,v}F^t(n+v).
\end{equation}
Therefore, we need to show that
\begin{equation}
\sum_{u\in\B_{k[i]}} p_{k[i],u}F^t(n+ e_i+u) - \sum_{v\in\B_{k}} p_{k,v}F^t(n+v) = \sum_{j=1,2}\sum_{u\in\B_{k}}c_{i,k,j,u}D_j^t(n+u).
\end{equation}

The result follows by
\begin{align}
\sum_{u\in\B_{k[i]}} p_{k[i],u}F^t(n+& e_i+u) - \sum_{v\in\B_{k}} p_{k,v}F^t(n+v)  \notag \\
=& \sum_{w\in \B_k \cup (\B_{k[i]}+e_i)} \left[ \Iset{w}{\B_{k[i]}+e_i}p_{k[i],w-e_i} - \Iset{w}{\B_k}p_{k,w}  \right] F^t(n+w) \\
=&\ \sum_{\substack{w\in \B_{k}\cup (\B_{k[1]}+e_1)\\ \cup(\B_{k[2]}+e_2)}}  \left[ \Iset{w}{\B_{k[i]}+e_i}p_{k[i],w-e_i} - \Iset{w}{\B_k}p_{k,w}  \right] F^t(n+w) \label{eq:dynDs3} \\
=&\ \sum_{\substack{w\in \B_{k}\cup (\B_{k[1]}+e_1)\\ \cup(\B_{k[2]}+e_2)}} \big[ \Iset{w-e_1}{\B_{k}}c_{i,k,1,w-e_1} - \Iset{w}{\B_{k}}c_{i,k,1,w} \notag \\
  &   + \Iset{w-e_2}{\B_{k}}c_{i,k,2,w-e_2} - \Iset{w}{\B_{k}}c_{i,k,2,w} \big] F^t(n+w) \label{eq:dynDs4} \\
=&\ \sum_{j=1,2}\sum_{u\in\B_{k}}c_{i,k,j,u}\left[F^t(n+u+e_j) - F^t(n+u)\right] \label{eq:dynDs5} \\
=&\ \sum_{j=1,2}\sum_{u\in\B_{k}}c_{i,k,j,u}D_j^t(n+u), \label{eq:dynDs6}
\end{align}
where: ---~\eqref{eq:dynDs3} holds because we extend the summation over $w$ for which $\Iset{w}{\B_{k[i]}+e_i}=\Iset{w}{\B_k}=0$, ---~\eqref{eq:dynDs4} follows directly from Assumption~\ref{ass:a}, ---~\eqref{eq:dynDs5} is an immediate consequence of~\eqref{eq:dynDs4} and finally, ---~\eqref{eq:dynDs6} follows by definition of $D_j^t$. \Halmos
\end{proof}

\begin{table}
\centering
\begin{IEEEeqnarraybox}[\IEEEeqnarraystrutmode]{{eqnarrayCscript}'{eqnarrayCscript}'{eqnarrayCscript}'c'l}
{\displaystyle k }  & {\displaystyle j} & {\displaystyle u} & & c_{1,k,j,u} \\ \hline
1 & 1 &  N_1 & & p_{1,u} \\
2 & 1 & \{d_1,e_1,d_2\} & & p_{4,u} \\
2 & 1 & e_2  & & p_{4,e_2}-p_{2,d_1}+c_{1,2,1,d_1} \\
2 & 1 & 0  & & p_{4,0}-p_{2,e_1}+c_{1,2,1,e_1}  \\
2 & 1 & -e_2 & & p_{4,\meb}-p_{2,d_2}+c_{1,2,1,d_2}  \\
2 & 2 & 0  & & p_{4,\mdb}-p_{2,e_2}+c_{1,2,1,e_2} \\
2 & 2 & -e_2  & & p_{4,\mea}-p_{2,0} + c_{1,2,2,0} + c_{1,2,1,0} \\
3 & 1 & \{e_1,d_1\} & & p_{1,u} \\
3 & 1 & e_2  & & p_{1,e_2}-p_{3,d_1}+c_{1,3,1,d_1} \\
3 & 1 & 0 & & p_{1,0}-p_{3,e_1}+c_{1,3,1,e_1} \\
3 & 2 & 0 & & p_{1,\mdb}-p_{3,e_2}+c_{1,3,1,e_2} \\
4 & 1 & N_4 & & p_{4,u}
\end{IEEEeqnarraybox}
%

\vspace{5mm}
\begin{IEEEeqnarraybox}[\IEEEeqnarraystrutmode]{{eqnarrayCscript}'{eqnarrayCscript}'{eqnarrayCscript}'c'l}
{\displaystyle k}  & {\displaystyle j} & {\displaystyle u} & & c_{2,k,j,u} \\ \hline
1 & 2 &  \{d_1,e_2,\mdb\} & & p_{4,u} \\
1 & 2 &  e_1 & & p_{4,e_1}-p_{1,d_1}+c_{2,1,2,d_1} \\
1 & 2 &  0 & & p_{4,0}-p_{1,e_2}+c_{2,1,2,e_2} \\
1 & 2 &  -e_1 & & p_{4,\mea}-p_{1,\mdb}+c_{2,1,2,\mdb} \\
1 & 1 &  0 & & p_{4,d_2}-p_{1,e_1}+c_{2,1,2,e_1} \\
1 & 1 &  -e_1 & & p_{4,\meb}-p_{1,0}+c_{2,1,1,0}+c_{2,1,2,0} \\
2 & 2 &  N_2 & & p_{2,u} \\
3 & 2 &  \{d_1,e_2\} & & p_{2,u} \\
3 & 2 &  e_1 & & p_{2,e_1}-p_{3,d_1}+c_{2,3,2,d_1} \\
3 & 2 &  0 & & p_{2,0}-p_{3,e_2}+c_{2,3,2,e_2} \\
3 & 1 &  0 & & p_{2,d_2}-p_{3,e_1}+c_{2,3,2,e_1} \\
4 & 2 &  N_4 & & p_{4,u}
\end{IEEEeqnarraybox}
\caption{Values for constants $c_{i,k,j,u}$.}
\label{table:constants}
\end{table}

Next, we show how to find coefficients $c_{i,k,j,u}$ that satisfy Assumption~\ref{ass:a}. Note that the constraints given in~\eqref{eq:constantsconditionsmain} of Assumption~\ref{ass:a} can be interpreted as a flow problem in which the variable $c_{i,k,j,u}$ is the amount of flow assigned to the `edge' from $n+u$ to $n+u+e_j$ and the RHS of~\eqref{eq:constantsconditionsmain} is the demand at `vertex' $w$. It is not necessary to solve this problem for each random walk at hand. Instead, we formulate below a solution by giving values of the constants $c_{i,k,j,u}$ in terms of the transition probabilities of the random walk. The result, which states that it is possible to satisfy Assumption~\ref{ass:a} is readily verified and, therefore, stated without proof.
\begin{theorem}
If $c_{i,k,j,u}$ are chosen according to Table~\ref{table:constants} then Assumption~\ref{ass:a} holds. 
\end{theorem}

Note, that the values of $c_{i,k,j,u}$ as given in Table~\ref{table:constants} are not the only values for which~\eqref{eq:dynD} is satisfied. We have chosen to present Theorem~\ref{th:main} in terms of constants that can be stated concisely and that are universal in the sense that they are a simple function of the transition probabilities that define the random walk. It would be of interest to include an optimization over these constants in the optimization problems that will be stated below. However, while the constraints~\eqref{eq:dynD} themselves are linear, the overall optimization problem would be non-linear. Therefore, it is outside the scope of the current work.

Next, we present a set of linear constraints on the functions $A_i$ and $B_i$ that ensure~\eqref{eq:biasboundsrequired}.
\begin{lemma} \label{lem:biasbounds}
If $A_i:\SS\to[0,\infty)$ and $B_i:\SS\to[0,\infty)$, $i=1,2$ satisfy
\begin{gather}
F(n+e_i)-F(n) + \sum_{j=1,2}\sum_{u\in\B_{k}} \max\{-c_{i,k,j,u}A_i(n+u), c_{i,k,j,u}B_i(n+u)\} \leq B_i(n),  \label{eq:DiboundA} \\
F(n)-F(n+e_i) + \sum_{j=1,2}\sum_{u\in\B_{k}} \max\{-c_{i,j,k,u}B_i(n+u), c_{i,j,k,u}A_i(n+u)\} \leq A_i(n), \label{eq:DiboundB}
\end{gather}
for all $n\in\S$ and $k=k(n)$. Then
\begin{equation}
-A_i(n) \leq D_i^t(n) \leq B_i(n),
\end{equation}
for $i=1,2$, $n\in\S$ and $t\geq 0$.
\end{lemma}
\begin{proof}{Proof}
We use induction over $t$. Since $A_i(n)$ and $B_i(n)$ are non-negative and $D_i^0(n)=0$ the bounds hold at $t=0$. Next, assume that $A_i(n)\leq D_i^t(n)\leq B_i(n)$ for both $i=1$ and $i=2$ at $t\geq 0$. Then
\begin{IEEEeqnarray}{rCl}
D_i^{t+1}(n) &=& F(n+e_i)-F(n) + \sum_{j=1,2}\sum_{u\in\B_{\k}}c_{i,k,j,u}D_j^t(n+u) \notag \\
&\leq&  F(n+e_i)-F(n) + \sum_{j=1,2}\sum_{u\in\B_{\k(n)}}\max\{-c_{i,k,j,u}A_i(n+u), c_{i,k,j,u}B_i(n+u)\} \notag \\
&\leq& B_i(n),
\end{IEEEeqnarray}
where the first equality follows from Lemma~\ref{lem:dynD}, the first inequality from the induction hypothesis and the last inequality from~\eqref{eq:DiboundA}. In the other direction we have
\begin{IEEEeqnarray}{rCl}
D_i^{t+1}(n) &\geq&  F(n+e_i)-F(n) + \sum_{j=1,2}\sum_{u\in\B_{\k(n)}}\min\{-c_{i,j,k,u}A_i(n+u), c_{i,j,k,u}B_i(n+u)\} \notag \\
&\geq& -A_i(n),
\end{IEEEeqnarray}
which follows from Lemma~\ref{lem:dynD}, the induction hypothesis and~\eqref{eq:DiboundB}. \Halmos
\end{proof}

\subsection{Main result: An error bound without bias terms} \label{ssec:errorbound}
Combining the results from the previous subsections leads to the following optimization problem. Like Problem~\ref{pr:stepa}, this problem provides an upper bound on $\FF$. A precise formulation of this result is given below.
\begin{problem} \label{pr:stepb}
\begin{align}
 \text{minimize}\ &\sum_{n\in\S}\left[\bar F(n) + G(n)\right]\bar \pi(n), \label{eq:stepbobj} \\
\text{subject to}\
  &\bigg| \bar F(n) - F(n) + \sum_{i=1,2}\left(q_{k(n),e_i}E_i(n) + q_{k(n),\mei}E_i(n-e_i) \right)\bigg| \leq G(n), \label{eq:stepbfirst} \\[.4em]
  & -A_i(n) \leq E_i(n) \leq B_i(n), \label{eq:stepbsecond} \\[.5em]
  &  F(n+e_i)-F(n)  + \sum_{j=1,2}\sum_{u\in\mathrlap{\B_{k(n)}}} \max\{-c_{i,k(n),j,u}A_i(n+u), c_{i,k(n),j,u}B_i(n+u)\} \leq B_i(n), \label{eq:stepbbiasbound1}  \\[.4em]
  & F(n)-F(n+e_i) + \sum_{j=1,2}\sum_{u\in\mathrlap{\B_{k(n)}}} \max\{-c_{i,j,k(n),u}B_i(n+u), c_{i,j,k(n),u}A_i(n+u)\}\leq A_i(n), \label{eq:stepbbiasbound2} \\[.4em]
  & \bar F(n)\geq 0, G(n)\geq 0, A_i(n)\geq 0, B_i(n)\geq 0,\qquad\text{for }n\in\S, i\in\{1,2\}. \label{eq:stepblast}
\end{align}
\end{problem}

The next theorem provides the main contribution of the current paper. As indicated at the start of this section, we will give the generalized result for the case that $R$ and $\bar R$ can have non-equal transition probabilities for transitions in arbitrary directions in Section~\ref{sec:generalize}.
\begin{theorem} \label{th:main}
Let $q_{k,u}=0$ if $u\not\in\{-e_1,e_1,-e_2,e_2,0\}$. Finally,
let $\FF^*$ denote the optimal value of Problem~\ref{pr:stepb}. Then $\FF\leq\FF^*$.
\end{theorem}
\begin{proof}{Proof:}
By Lemmas~\ref{lem:dynD} and~\ref{lem:biasbounds} and constraints~\eqref{eq:stepbbiasbound1} and~\eqref{eq:stepbbiasbound2} it follows that 
\begin{equation}
   -A_i(n) \leq D_i(n) \leq B_i(n).
\end{equation}
Now it follows from constraints~\eqref{eq:stepbfirst} and~\eqref{eq:stepbsecond} and from the fact that $q_{k,u}=0$ if $u\not\in\{-e_1,e_1,-e_2,e_2,0\}$ that
\begin{equation}
\Big| \bar F(n) - F(n) + \sum_{u\in\B_{k(n)}}q_{k(n),u}D_u^t(n) \Big| \leq G(n).
\end{equation}
Finally, the result follows from Theorem~\ref{th:error}. \Halmos
\end{proof}

\subsection{A finite linear program} \label{ssec:finite}
The final step is to reduce Problem~\ref{pr:stepb} to a linear program with a finite number of variables and a finite number of constraints. In the remainder, we will refer to such a linear program as a finite linear program. So far, besides constraints~\eqref{eq:stepbfirst}--\eqref{eq:stepblast} we have not put any restrictions on the functions $\bar F, G, A_1, A_2, B_1$ and $B_2$. In the most general case, each of these functions is specificied by one variable for each element in the state space, \ie we have a linear program with countably many variables. Next, we put additional constraints on these functions, such that the total number of variables is finite. Recall from Section~\ref{sec:prelim} that the performance measures that we consider are induced by componentwise linear functions, \ie $F$ is $\C$-linear. Moreover, the transition probabilities of $R$ and $\bar R$ are homogeneous within each component. Therefore, we restrict our attention to functions $\bar F, G, A_1, A_2, B_1$ and $B_2$ that are $\C$-linear.

A $\C$-linear function can be specified by means of $8$ coefficients, see~\eqref{eq:clinear}.
We will demonstrate below that constraints~\eqref{eq:stepbfirst}--\eqref{eq:stepblast} are equivalent to a finite number of linear constraints in the coefficients that define the $\C$-linear functions $\bar F, G, A_1, A_2, B_1$ and $B_2$. In addition we show that objective function~\eqref{eq:stepbobj} is linear in these coefficients.

Before giving a complete description of the reduction to a finite number of constraint we give an overview of the main ideas. The key idea that enables reduction to a finite number of constraints is that each of the constraints~\eqref{eq:stepbfirst}--\eqref{eq:stepblast} can be reformulated as a sign constraint on a function from a class that will be specified below. We will see that in this class of functions, sign constraints are equivalent to a finite number of linear constraints. To illustrate the idea, we give an example for linear function $h:\SS\to\RR$, $h(n)=h_0+h_1 n_1+h_2 n_2$, for which the condition $h(n)\geq 0$ for all $n\in\SS$ is obviously equivalent to the three constraints $h_0\geq 0$, $h_1\geq 0$ and $h_2\geq 0$.

\begin{figure}
\centering
\begin{tikzpicture}[scale=1]
\setcounter{cA}{3}
\setcounter{cB}{3}
\setlength{\lA}{4mm}
\setlength{\lB}{3cm}
\setlength{\lC}{1.5mm}
\tikzstyle{vertex}=[draw,circle,fill,color=black!30!white,outer sep=0pt, inner sep=0pt, minimum size=3pt]
\tikzstyle{part}=[rounded corners]
\tikzstyle{partlabel}=[anchor=west,xshift=-\lA+1mm,yshift=-\lC,inner sep=0pt, outer sep=0pt]
\foreach \x in {0,...,\thecA}
 	\foreach \y in {0,...,\thecB}
 		\node[vertex] (x\x y\y) at (\x,\y) {}; 
\draw[part] (-\lA,-\lA) rectangle (\lA,\lA);
\draw[part] (-\lA+1cm,-\lA) rectangle (\lA+1cm,\lA);
\draw[part] (-\lA,-\lA+1cm) rectangle (\lA,\lA+1cm);
\draw[part] (-\lA+1cm,-\lA+1cm) rectangle (\lA+1cm,\lA+1cm);
\draw[part] (\lA+\lB,-\lA+1cm) -- (-\lA+2cm,-\lA+1cm) -- (-\lA+2cm,\lA+1cm) -- (\lA+\lB,\lA+1cm);
\draw[part] (-\lA+1cm,\lA+\lB) -- (-\lA+1cm,-\lA+2cm) -- (\lA+1cm,-\lA+2cm) -- (\lA+1cm,\lA+\lB);
\draw[part] (\lA+\lB,-\lA) -- (-\lA+2cm,-\lA) -- (-\lA+2cm,\lA) -- (\lA+\lB,\lA);
\draw[part] (-\lA,\lA+\lB) -- (-\lA,-\lA+2cm) -- (\lA,-\lA+2cm) -- (\lA,\lA+\lB);
\draw[part] (\lA+\lB,-\lA+2cm) -- (-\lA+2cm,-\lA+2cm) -- (-\lA+2cm,\lA+\lB);
\node[partlabel] at (0,0) {\fm{\T_1}};
\node[partlabel] at (1,0) {\fm{\T_2}};
\node[partlabel] at (2.5,0) {\fm{\T_3}};
\node[partlabel] at (0,1) {\fm{\T_4}};
\node[partlabel] at (1,1) {\fm{\T_5}};
\node[partlabel] at (2.5,1) {\fm{\T_6}};
\node[partlabel] at (0,2.5) {\fm{\T_7}};
\node[partlabel] at (1,2.5) {\fm{\T_8}};
\node[partlabel] at (2.5,2.5) {\fm{\T_9}};
 \end{tikzpicture}	
\caption{The $\T$ partition of $\S$.}
\label{fig:partitionT}
\end{figure}
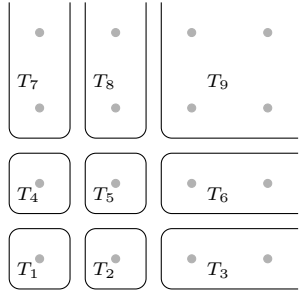

The linear function from the previous examples captures most of the characteristics from the general case. However, the class of linear functions is not rich enough for our purposes. Since we start with functions that are componentwise linear over the $\C$ partition, it is obvious that we need to consider at least componentwise linear functions. The $\C$ partition is, however, not fine enough. Indeed, if $A_i(n)$ is $\C$-linear, then $A_i(n+u)$, as occuring in, \eg \eqref{eq:stepbbiasbound1},  is not $\C$-linear. Therefore, we introduce a finer partition of the state space on which $A_i(n+u)$ is componentwise linear. We will show below that all functions that we need to consider are componentwise linear over this finer partition. Let 
\begin{equation}
\begin{array}{l@{\ }l@{\ }l}
\T_1=\{(0,0)\},&
\T_4=\{(0,1)\},&
\T_7=\{0\}\!\!\times\!\!\{2,3,\dots\},\\
\T_2=\{(1,0)\},&
\T_5=\{(1,1)\},&
\T_8=\{1\}\!\!\times\!\!\{2,3,\dots\},\\
\T_3=\{2,3,\dots\}\!\!\times\!\!\{0\}, &
\T_6=\{2,3,\dots\}\!\!\times\!\!\{1\}, &
\T_9=\{2,3,\dots\}\!\!\times\!\!\{2,3,\dots\},
\end{array}
\end{equation}
as illustrated in Figure~\ref{fig:partitionT}. In accordance with the definition for the $\C$ partition, let $\t:\S\to\{1,\dots,9\}$ be defined through $n\in\T_{t(n)}$. We refer to functions that are linear in each of the sets $\T_1,\dots,\T_9$ as $\T$-linear. A $\T$-linear function $h:\S\to\mathbb{R}$ is defined through a set of coefficients $h_{t,i}$, $1\leq t\leq 9$, $i=0,1,2$, \ie
\begin{equation}
 h(n) = h_{t(n),0} +h_{t(n),1}\na + h_{t(n),2}\nb.
\end{equation}

Next, we present three simple results, the proofs of which are straightforward and omitted. 
\begin{lemma} \label{lem:Cshift}
Let $H:\S\to\RR$ be $\C$-linear and let $u\in\N$. Define $\hat H:\SS\to\RR$ as $\hat H(n)=H(n+u)$, if $n+u\in\SS$, and $\hat H(n)=0$ otherwise. Then $\hat H$ is a $\T$-linear function.
\end{lemma}

\begin{lemma} \label{lem:Hfinite}
Let $H:\S\to\RR$ be $\T$-linear, $H(n) = h_{t(n),0} +h_{t(n),1}\na + h_{t(n),2}\nb$. Then $H(n)\geq 0$ for all $n\in\S$ if and only if the coefficients $h_{t,i}$ satisfy the linear constraints 
\begin{IEEEeqnarray*}{rCl'rCl+rCl'rCl} 
h_{1,0}&\geq&0, &   &&   & h_{2,0}+h_{2,1}&\geq& 0, &   &&  \\
h_{3,0}+2h_{3,1}&\geq& 0, & h_{3,1}&\geq& 0, & h_{4,0}+h_{4,2}&\geq& 0, & && \\
h_{5,0}+h_{5,1}+h_{5,2}&\geq&0, & && & h_{6,0}+2h_{6,1}+h_{6,2}&\geq& 0, & h_{6,1}&\geq& 0, \\
h_{7,0}+2h_{7,2}&\geq& 0, & h_{7,2}&\geq& 0,  & h_{8,0}+h_{8,1}+2h_{8,2}&\geq& 0, & h_{8,2}&\geq& 0,
\end{IEEEeqnarray*}
and 
\begin{equation*}
h_{9,0}+2h_{9,1}+2h_{9,2}\geq 0,\quad h_{9,1}\geq 0,\quad h_{9,2}\geq 0. 
\end{equation*}
\end{lemma}

\begin{lemma} \label{lem:objlinear}
If $H:\S\to\RR$ is $\C$-linear, with $H(n) = h_{t(n),0} +h_{t(n),1}\na + h_{t(n),2}\nb$, and $\bar\pi(n)=\prod_{i=1,2}(1-\ri)\ri^{\ni}$, then 
\begin{multline}
\sum_{n\in\S} H(n)\bar\pi(n) =h_{3,0}(1-r_1)(1-r_2) + r_1(1-r_2)\left(h_{1,0}+\frac{h_{1,1}}{1-r_1}\right)  \\
+ (1-r_1)r_2\left(h_{2,0}+\frac{h_{2,2}}{1-r_2}\right) +  r_1r_2\left(h_{4,0}+\frac{h_{4,1}}{1-r_1}+\frac{h_{4,2}}{1-r_2}\right).
\end{multline}
\ie $\sum_{n\in\S} H(n)\bar\pi(n)$ is a linear function in the variables $h_{t,i}$.
\end{lemma}

From Lemmas~\ref{lem:Cshift}--\ref{lem:objlinear} it is clear that Problem~\ref{pr:stepb} can be reduced to a finite linear program by imposing the additional constraint that the functions $\bar F, G, A_1, A_2, B_1$ and $B_2$ are $\C$-linear. The formal result is presented below for completeness.
\begin{theorem} \label{th:finite}
Problem~\ref{pr:stepb}, with the additional constraint that $\bar F, G, A_1, A_2, B_1$ and $B_2$ are $\C$-linear, is a finite linear program.
\end{theorem}
\begin{proof}{Proof:}
First observe that even though constraints~\eqref{eq:stepbfirst}--\eqref{eq:stepblast} themselves are not linear, they can readily be replaced by constraints that are linear in the functions$\bar F, G, A_1, A_2, B_1$ and $B_2$. Next, it follows from Lemma~\ref{lem:Cshift} that these constraints can be reduced to non-negativity of $\T$-linear functions. Note that the technical condition in Lemma~\ref{lem:Cshift}, $\hat H(n)=H(n+u)=0$ if $n+u\neq\S$, does not come into play, since all expressions involve only $u\in\B_{k(n)}$. From Lemma~\ref{lem:Hfinite} it follows that non-negativity is equivalent to a finite number of constraints in the coefficients that constitute these functions. Finally, it follows from Lemma~\ref{lem:objlinear} that the objective function~\eqref{eq:stepbobj} is linear. \Halmos
\end{proof}

It is possible to craft the linear constraints in the coefficients of the functions $\bar F, G, A_1, A_2, B_1$ and $B_2$ for the finite linear program by hand. This is, however, a tedious and error-prone process. A more convenient method of generating the finite linear program is by making use of a mathematical programming language like AMPL~\cite{ampl}.  Indeed, the representation of Problem~\ref{pr:stepb} together with Lemmas~\ref{lem:Cshift}--\ref{lem:objlinear} straightforwardly leads to an implementation in a mathematical programming language.

%
%
%
\section{Generalization and variations} \label{sec:generalize}
In this section we present three additional results. First we present a method to establish a comparison result, \ie an ordering, on $R$ and $\bar R$. After that we generalize Theorem~\ref{th:main} from Section~\ref{sec:result} to include the case that the transition rates of $R$ and $\bar R$ are different for transitions that are not along the unit directions. Finally, we present results on establishing lower bounds on performance.

\subsection{Comparison result}
The results that have been presented in Section~\ref{sec:result} are based on the error bound result by van Dijk Theorem~\ref{th:error}. The next result by van Dijk, as found in, for instance~\cite{vandijk11inbook}, provides a direct comparison between two random walks.
\begin{theorem}[\cite{vandijk11inbook}] \label{th:comparisonnico}
Let $\bar F:\S\to [0,\infty)$ satisfy
\begin{equation} \label{eq:constrthcomp}
\bar F(n) - F(n) + \sum_{u\in\B_{k(n)}}q_{k(n),u}D_u^t(n) \geq 0,
\end{equation}
for all $n\in\S$ and $t\geq 0$.
Then
\begin{equation*}
 \FF\ \leq\ \sum_{n\in\S}\bar F(n)\bar \pi(n).
\end{equation*}
\end{theorem}
The relevance of the above result is twofold. First, there are cases in which Theorem~\ref{th:comparisonnico} results in a better upper bound on $\FF$ than Theorem~\ref{th:error}. In Section~\ref{sec:example} we will provide some examples. It should also be noted that Theorem~\ref{th:comparisonnico} is not universally better than Theorem~\ref{th:comparisonnico}. In fact there are examples in which there are no solutions to Theorem~\ref{th:comparisonnico}, but for which Theorem~\ref{th:error} is valid.

The second use of Theorem~\ref{th:comparisonnico} stems from the fact that useful results can be deducted without explicit knowledge of $\bar \pi$,  the invariant measure of the perturbed random walk. Indeed a comparison can be made between two systems directly. This can be useful, for instance, in analyzing the effect of changing certain parameters, like specific transition probabilities.

The first variation of Problem~\ref{pr:stepb} that we consider is a straighforward application of Theorem~\ref{th:comparisonnico}. The variables in the optimization problem below are the functions $\bar F, A_1, A_2, B_1$ and $B_2$. Since the aim is no longer to obtain a bound on the modulus of the LHS of~\eqref{eq:constrthcomp}, there is no function $G$.  
\begin{problem} \label{pr:comparisona}
\begin{align}
 \text{minimize}\ &\sum_{n\in\S}\bar F(n) \bar \pi(n), \label{eq:stepbobj} \\
\text{subject to}\
  &\bar F(n) - F(n) + \sum_{i=1,2}\left(q_{k(n),e_i}E_i(n) + q_{k(n),\mei}E_i(n-e_i) \right) \geq 0, \label{eq:comparisonfirst} \\[.4em]
  & -A_i(n) \leq E_i(n) \leq B_i(n), \label{eq:stepbsecond} \\[.5em]
  &  F(n+e_i)-F(n)  + \sum_{j=1,2}\sum_{u\in\mathrlap{\B_{k(n)}}} \max\{-c_{i,k(n),j,u}A_i(n+u), c_{i,k(n),j,u}B_i(n+u)\} \leq B_i(n), \label{eq:stepbbiasbound1}  \\[.4em]
  & F(n)-F(n+e_i) + \sum_{j=1,2}\sum_{u\in\mathrlap{\B_{k(n)}}} \max\{-c_{i,j,k(n),u}B_i(n+u), c_{i,j,k(n),u}A_i(n+u)\}\leq A_i(n), \label{eq:stepbbiasbound2} \\[.4em]
  & \bar F(n)\geq 0, A_i(n)\geq 0, B_i(n)\geq 0,\qquad\text{for }n\in\S, i\in\{1,2\}. \label{eq:comparisonlast}
\end{align}
\end{problem}

 The next corollary is an immediate consequence of Theorems~\ref{th:main} and~\ref{th:comparisonnico}.
\begin{corollary} \label{cor:comparisona}
Let $q_{k,u}=0$ if $u\not\in\{-e_1,e_1,-e_2,e_2,0\}$ and let $\FF^*$ denote the optimal value of Problem~\ref{pr:comparisona}. Then $\FF\leq\FF^*$.
\end{corollary}

The difference between Problems~\ref{pr:stepb} and~\ref{pr:comparisona} is small in the sense that both problems require upper and lower bounds on $D_i^t(n)$. There are cases it is not possible to find such upper and lower bounds in which case neither Problem~\ref{pr:stepb} nor Problem~\ref{pr:comparisona} has a feasible solution. However, in some of these cases it might still be possible to obtain a result on the sign of $D_i^t(n)$. Together with the sign of $q_{k,e_i}$ and $q_{k,\mei}$ this could be used to establish~\eqref{eq:constrthcomp} and obtain a comparison result.

\subsection{Arbitrary perturbations}
In Section~\ref{sec:result} we derived an error bound result for the case that the perturbations from $R$ to $\bar R$ were along the unit directions only, \ie $q_{k,u}=0$ if $u\not\in\{-e_1,e_1,-e_2,e_2,0\}$. In this subsection we extend this result to arbitary perturbations, \ie arbitary $q_{k,u}$. The method we use for this generalization is to use the bounds on the bias terms $D_{e_1}^t=D_1^t$ and $D_{e_2}^t=D_2^t$ that are obtained from Constraints~\eqref{eq:stepbbiasbound1} and~\eqref{eq:stepbbiasbound2} to construct bounds on the bias terms $D_u^t$ in the other directions ,\ie for $u\neq\{e_1,e_2\}$. In order to prevent confusion, in this section we will refrain from using the notation $D^t_i$, $A_i$ and $B_i$. Instead we will use the full forms $D^t_{e_i}$, $A_{e_i}$ and $B_{e_i}$.

For the purpose of bounding $D_u^t$  we introduce functions $A_u$ and $B_u$ for each $u\in\N$. In similar spirit to previous considerations the aim is to achieve
\begin{equation}
 -A_u(n)\leq D_u^t(n)\leq B_u(n).
\end{equation}
In Section~\ref{sec:result} we have obtained bounds on $A_{e_i}$ and $B_{e_i}$, for $i=1,2$. We present a construction to reuse these bounds and obtain the desired result on $A_u$ and $B_u$ for all $u\in\NN$.

Before giving the general construction of the functions $A_u$, $B_u$, we provide an example for the case that $u=-d_2$. First note that $D_{\mdb}^t(n)$ is only defined if $n-d_2\in\S$. Now, using the fact that $-d_2=-e_1+e_2$ we have
\begin{align}
D^t_{\mdb}(n)
&= F^t(n-d_2) - F^t(n) \\
&= \left(F^t(n-e_1+e_2) - F^t(n-e_1)\right) - \left(F^t(n-e_1+e_1) - F^t(n-e_1)\right) \\
&= D^t_{e_2}(n-e_1) - D^t_{e_1}(n-e_1). 
\end{align}
By constructing $A_{\mdb}$ and $B_{\mdb}$ on domain $S\cap(S+d_2)$ as
\begin{align}
A_{\mdb}(n) &= A_{e_2}(n) + B_{e_1}(n-e_1), \\
B_{\mdb}(n) &= B_{e_2}(n) + A_{e_1}(n-e_1),
\end{align}
we achieve $-A_{\mdb}(n)\leq D^t_{\mdb}(n)\leq B_{\mdb}(n)$ as required.
The general construction for arbitrary $u\in\N$ is given in the next result. Since $D_u^t(n)=F^t(n+u)-F^t(n)$ is not defined if $n+u\not\in\S$ we need to take some care in defining the domain of the functions $A_u$ and $B_u$.
\begin{lemma} \label{lem:ABu}
Consider for $u\in\N\setminus\{e_1,e_2\}$ the functions $A_u: S\cap(S-u)\to[0,\infty) $ and $B_u: S\cap(S-u)\to[0,\infty)$ defined as
\begin{IEEEeqnarray*}{rCl}
A_{u}(n) &=& \II{u_1=1}A_{e_1}(n) + \II{u_1=-1}B_{e_1}(n-e_1) + \II{u_2=1}A_{e_2}(n+u_1e_1) \\
  && + \II{u_2=-1}B_{e_2}(n+u_1e_1+u_2e_2), \\
B_{u}(n) &=& \II{u_1=1}B_{e_1}(n) + \II{u_1=-1}A_{e_1}(n-e_1) + \II{u_2=1}B_{e_2}(n+u_1e_1) \\
  &&  + \II{u_2=-1}A_{e_2}(n+u_1e_1+u_2e_2),
\end{IEEEeqnarray*}
where, for $i=1,2$, $A_{e_i}: S\to[0,\infty)$ and $B_{e_i}: S\to[0,\infty)$.
If $-A_{e_i}(n)\leq D_{e_i}^t(n)\leq B_{e_i}(n)$, $i=1,2$, $t\geq 0$, then
\begin{equation}
-A_u(n)\leq D_u^t(n)\leq B_u(n),
\end{equation}
for all $u\in\N$, $n\in\S$ and all $t\geq 0$.
\end{lemma}
\begin{proof}{Proof}
The results follows directly from the observation that we can write
\begin{IEEEeqnarray*}{rCl'rCl}
D_{\mea}(n) &=& -D_{e_1}(n-e_1), & D_{\meb}(n) &=& -D_{e_2}(n-e_2), \\
D_{d_1}(n) &=& D_{e_1}(n) + D_{e_2}(n+e_1), & D_{\mda}(n) &=& -D_{e_1}(n-e_1) - D_{e_2}(n-e_1-e_2), \\
D_{d_2}(n) &=& D_{e_1}(n) - D_{e_2}(n+e_1-e_2), & D_{\mdb}(n) &=& -D_{e_1}(n-e_1) + D_{e_2}(n-e_1),
\end{IEEEeqnarray*}
\ie that
\begin{multline*}
D_{u}(n) = \II{u_1=1}D_{e_1}(n) - \II{u_1=-1}D_{e_1}(n-e_1) + \II{u_2=1}D_{e_2}(n+u_1e_1) \\
  - \II{u_2=-1}D_{e_2}(n+u_1e_1-e_2). \Halmos
\end{multline*}
\end{proof}

Next, we provide the natural extension of Problem~\ref{pr:stepb} that includes the bounds on the bias terms in all directions. Like Problem~\ref{pr:stepb} the optimal value of the problem provides an upper bound on $\FF$. A formal statement of this result is given below. 
\begin{problem} \label{pr:arbitrary}
\begin{align}
 \text{minimize}\ &\sum_{n\in\S}\left[\bar F(n) + G(n)\right]\bar \pi(n), \\
\text{subject to}\
  &\bigg| \bar F(n) - F(n) + \sum_{u\in\B_{k(n)}}q_{k(n),u}E_u(n) \bigg| \leq G(n),\quad \text{for }n\in\S,  \\[.5em]
  & -A_u(n) \leq E_u(n) \leq B_u(n),\quad \text{for }n\in\S\cap(S-u), u\in\B, \\[.5em]
  &  A_{u}(n) = \II{u_1=1}A_{e_1}(n)  + \II{u_1=-1}B_{e_1}(n-e_1) + \II{u_2=1}A_{e_2}(n+u_1e_1) \notag \\  
  &      \hspace{28mm} + \II{u_2=-1}B_{e_2}(n+u_1e_1+u_2e_2), \quad \text{for }n\in\S\cap(S-u), u\in\N, \\[.5em]
  &  B_{u}(n) = \II{u_1=1}B_{e_1}(n)  + \II{u_1=-1}A_{e_1}(n-e_1) + \II{u_2=1}B_{e_2}(n+u_1e_1) \notag \\ 
  & \hspace{28mm} + \II{u_2=-1}A_{e_2}(n+u_1e_1+u_2e_2),  \quad \text{for }n\in\S\cap(S-u), u\in\N, \\[.5em]
  &  F(n+e_i)-F(n)  + \sum_{j=1,2}\sum_{u\in\B_{k(n)}} \max\{-c_{i,k(n),j,u}A_{e_j}(n+u), c_{i,k(n),j,u}B_{e_j}(n+u)\}  \notag \\[-3mm]
     &\hspace{83mm}\leq B_{e_i}(n),\quad \text{for }n\in\S,i\in\{1,2\},   \\[.5em]
  & F(n)-F(n+e_i) + \sum_{j=1,2}\sum_{u\in\B_{k(n)}} \max\{-c_{i,j,k(n),u}B_{e_j}(n+u), c_{i,j,k(n),u}A_{e_j}(n+u)\} \notag \\[-3mm]
  & \hspace{83mm} \leq A_{e_i}(n), \quad \text{for }n\in\S,i\in\{1,2\},  \\[.5em]
  & \bar F(n)\geq 0, G(n)\geq 0, A_u(n)\geq 0, B_u(n)\geq 0,\quad \text{for }n\in\S\cap(S-u), u\in\B.
\end{align}
\end{problem}
\begin{theorem} \label{th:arbitrary}
Let $\FF^*$ denote the optimal value of Problem~\ref{pr:arbitrary}. Then $\FF\leq\FF^*$.
\end{theorem}
\begin{proof}{Proof:}
Directly from Lemmas~\ref{lem:dynD}, \ref{lem:biasbounds} and~\ref{lem:ABu} and Theorem~\ref{th:error}.\Halmos
\end{proof}

Lemma~\ref{lem:ABu} and Problem~\ref{pr:arbitrary} provide one means of establishing a linear programming based error bound. An alternative approach is to directly extend Lemma~\ref{lem:dynD} to the case of bias terms in arbitrary directions. More precisely, this approach would involve finding constants $g_{u,k,v,w}$ for $u,v,w\in\B$, $k=1,\dots,4$ such that
\begin{equation} \label{eq:dynDarbitrary}
D_u^{t+1}(n) = F(n+u)-F(n) + \sum_{v\in\B}\sum_{w\in\B_{\k(n)}}g_{u,k(n),v,w}D_v^t(n+w),
\end{equation}
for $u\in\B$, $n\in\S$ and $t>0$. From~\eqref{eq:dynDarbitrary} we could then develop a generalization of Lemma~\ref{lem:biasbounds} and an alternative to Problem~\ref{pr:arbitrary}. Such an approach would not be hampered by any technical difficulties. However, it would also not provide any additional insights over Problem~\ref{pr:arbitrary}. Therefore, this approach is not pursued in the current paper.

\subsection{Lower bounds}
All results that have been presented in this paper so far deal with upper bounds on $\FF$. Corresponding lower bounds can trivially be obtained. We have for instance the following maximization problem and corollary to Theorems~\ref{th:error} and~\ref{th:main}.
\begin{problem} \label{pr:stepblow}
\begin{align}
& \text{maximize}\ \sum_{n\in\S}\left[\bar F(n) - G(n)\right]\bar \pi(n), \\
& \text{subject to Constraints~\eqref{eq:stepbfirst}--\eqref{eq:stepblast} of Problem~\ref{pr:stepb}.}
\end{align}
\end{problem}
\begin{corollary} \label{cor:mainlow}
Let $q_{k,u}=0$ if $u\not\in\{-e_1,e_1,-e_2,e_2,0\}$. Finally,
let $\FF^*$ denote the optimal value of Problem~\ref{pr:stepblow}. Then $\FF\geq\FF^*$.
\end{corollary}
In similar spirit a variation of Problem~\ref{pr:comparisona} and Corollary~\ref{cor:comparisona} can be obtained from the following corollary to Theorem~\ref{th:comparisonnico}.
\begin{corollary} \label{cor:comparisonnicolow}
Let $\bar F:\S\to [0,\infty)$ satisfy
\begin{equation} 
\bar F(n) - F(n) + \sum_{u\in\B_{k(n)}}q_{k(n),u}D_u^t(n) \leq 0,
\end{equation}
for all $n\in\S$ and $t\geq 0$.
Then
\begin{equation*}
 \FF\ \geq\ \sum_{n\in\S}\bar F(n)\bar \pi(n).
\end{equation*}
\end{corollary}

%
%
%
\section{Examples} \label{sec:example} \label{sec:example}
In this section we provide a number of examples that illustrate the use of the linear programming approach to obtaining error bounds. First we revisit the example from Section~\ref{sec:motivation}. Subsequently we will consider the case of coupled processors.

\subsection{Joint departures}
We continue with the example of a random walk with joint departures that was discussed in Section~\ref{sec:motivation}. In this section we will provide more extensive results on the performance of this random walk. We restrict our attention to the symmetric case that $\lambda_1=\lambda_2=\lambda$, $2\lambda+\mu=1$ and $\mu_1=\mu_2=\mu^*$, with $0<\mu^*\leq\mu$. The purpose of this section is to demonstrate the following: i) The performance bounds given in Proposition~\ref{prop:prelimbounds} can be improved, ii) The use of componentwise linear functions $A_i$, $B_i$ can significantly improve performance, iii) There are values of $\lambda$, $\mu$ and $\mu^*$ for which bounds cannot be obtained since the corresponding linear program does not have any feasible solutions, and finally iv) There are cases in which error bounds exist, but a comparison result cannot be obtained.

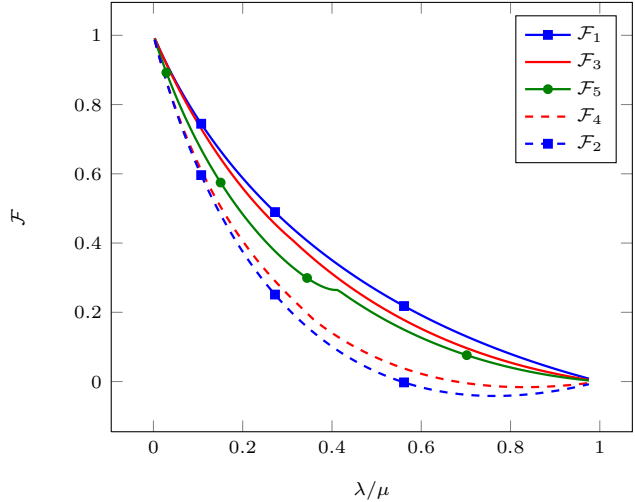
\begin{figure}
\centering
\begin{tikzpicture}
\begin{axis}[
  xlabel=$\lambda/\mu$,ylabel=$\FF$, 
  font=\scriptsize,
  legend style={
        cells={anchor=west},
        legend pos=north east,
       font=\scriptsize,
    }
]
\addplot[
  line width=.3mm,color=blue,
  mark=square*,mark repeat=32,mark phase=32,mark size=.5mm,mark options={solid}
  ]
table[
  header=false,x index=0,y index=2,
  ]
{matlab_figures/nc_empty.csv};
\addlegendentry{$\FF_1$};
\addplot[
  line width=.3mm,color=red,
  mark=none
  ]
table[
  header=false,x index=0,y index=4,
  ]
{matlab_figures/nc_empty.csv};
\addlegendentry{$\FF_3$};
\addplot[
  line width=.3mm,color=green!50!black,
  mark=*,mark repeat=32,mark phase=10,mark size=.5mm,mark options={solid}
  ]
table[
  header=false,x index=0,y index=8,
  ]
{matlab_figures/nc_empty.csv};
\addlegendentry{$\FF_5$};
\addplot[
  mark=none,line width=.3mm,color=red,dashed
  ]
table[
  header=false,x index=0,y index=3,
  ]
{matlab_figures/nc_empty.csv};
\addlegendentry{$\FF_4$};
\addplot[
  line width=.3mm,color=blue,dashed,
  mark=square*,mark repeat=32,mark phase=32,mark size=.5mm,mark options={solid}
  ]
table[
  header=false,x index=0,y index=1,
  ]
{matlab_figures/nc_empty.csv};
\addlegendentry{$\FF_2$};
\end{axis}
\end{tikzpicture}
\caption{Probability that the symmetric random walk with joint departures is empty, \ie $F(n)=\II{n=0}$. Upper bounds in solid lines, lower bounds in dashed lines.($\mu^*=0.4\mu$)}
\label{fig:ncempty}
\end{figure}

We first provide results for the performance measure that was considered in Section~\ref{sec:motivation}, the probability that the system is empty, \ie $F(n)=\II{n=0}$. Moreover we consider the perturbed random walk with $\bar\mu_1=\bar\mu_2=\mu/2$, again as in Section~\ref{sec:motivation}. Let $\FF_1$ and $\FF_2$ denote the values of the upper and lower bound, respectively, as given in Proposition~\ref{prop:prelimbounds} in Section~\ref{sec:motivation}. Moreover, let $\FF_3$ and $\FF_4$ denote the optimal values of Problems~\ref{pr:stepb} and~\ref{pr:stepblow}, respectively.
Finally, let $\FF_5$ denote the optimal value of Problem~\ref{pr:comparisona}, \ie the comparison result. The values of these bounds are illustrated in Figure~\ref{fig:ncempty} as a function of the system load $\lambda/\mu$. Recall from above that $\mu=1-2\lambda$. 

In Figure~\ref{fig:ncempty} we observe that the optimized bounds $\FF_3$ and $\FF_4$ are tighter than the bounds $\FF_1$ and $\FF_2$ that were manually derived in Section~\ref{sec:motivation}. Next, note that the comparison result of Problem~\ref{pr:comparisona} provides an even better upper bound. Observe, moreover, that the value of $\FF_5$ as given by Problem~\ref{pr:comparisona} consists of two piecewise smooth parts. The reason is the following. A more careful inspection of the optimal values of $\bar F$, $A_1$, $A_2$, $B_1$ and $B_2$ for Problem~\ref{pr:comparisona} reveils that the structure of the optimal $\bar F$ can have two forms depending on the value of $\lambda/\mu$. The final remark with respect to Figure~\ref{fig:ncempty} is that Problem~\ref{pr:stepblow} does not always provide a meaningful lower bound, \ie in our case it provides for some values of $\lambda/\mu$ a negative lower bound on a probability.

\begin{figure}
\centering
\begin{tikzpicture}
\begin{axis}[
  xlabel=$\lambda/\mu$,ylabel=$\FF$, 
  ymin=0,ymax=4,
  font=\scriptsize,
  legend style={
        cells={anchor=west},
        legend pos=north west,
       font=\scriptsize,
    }
]
\addplot[
  line width=.3mm,color=blue,
mark=square*,mark repeat=32,mark phase=32,mark size=.5mm,mark options={solid}
  ]
table[
  header=false,x index=0,y index=2,
  ]
{matlab_figures/nc_size.csv};
\addlegendentry{$\FF_1$};
\addplot[
  mark=none,line width=.3mm,color=red,
  mark=none
  ]
table[
  header=false,x index=0,y index=4,
  ]
{matlab_figures/nc_size.csv};
\addlegendentry{$\FF_3$};
\addplot[
  line width=.3mm,color=red,dashed,
  ]
table[
  header=false,x index=0,y index=3,
  ]
{matlab_figures/nc_size.csv};
\addlegendentry{$\FF_4$};
\addplot[
  line width=.3mm,color=blue, dashed,
  mark=square*,mark repeat=32,mark phase=32,mark size=.5mm,mark options={solid}
  ]
table[
  header=false,x index=0,y index=1,
  ]
{matlab_figures/nc_size.csv};
\addlegendentry{$\FF_2$};
\end{axis}
\end{tikzpicture}
\caption{Marginal first moment, \ie $F(n)=n_1$,  of the symmetric random walk with joint departures. Upper bounds in solid lines, lower bounds in dashed lines.($\mu^*=0.4\mu$)}
\label{fig:ncsize}
\end{figure}
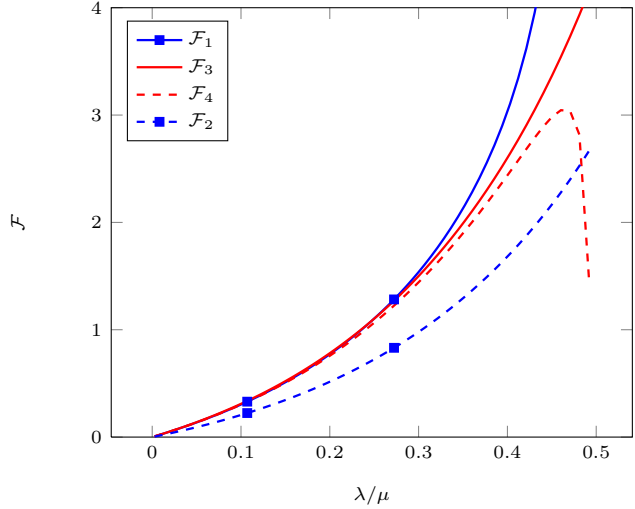

Next, we consider the performance measure $F(n)=n_1$, \ie $\FF$ is the first marginal moment in dimension $1$. Since we consider a completely symmetrical system this is equal to the first marginal moment in dimension $2$. Hence, we will simply refer to $\FF$ as the first marginal moment. In Figure~\ref{fig:ncsize} we have depicted various bounds on $\FF$ as a function of $\lambda/\mu$ for the case that $\mu^*=0.4\mu$. More precisely, the bounds in Figure~\ref{fig:ncsize} correspond to two different perturbed system. The first perturbed system that we consider is $\bar\mu_1=\mu-\mu^*$ and $\bar\mu_2=\mu^*$, leading to bounds $\FF_1$ and $\FF_2$. The second perturbed system has $\bar\mu_1=\mu^*$ and $\bar\mu_2=\mu-\mu^*$, leading to bounds $\FF_3$ and $\FF_4$. Upper bounds $\FF_1$ and $\FF_3$ are given by Problem~\ref{pr:comparisona}, lower bounds $\FF_2$ and $\FF_4$ by Problem~\ref{pr:stepblow}. The first thing to observe from Figure~\ref{fig:ncsize} is that the perturbed system that is considered can have significant impact on the tightness of the bounds that are derived. The second thing to note is that for larger values of $\lambda/\mu$ the bounds diverge. Inspection of the relevant linear programs reveals that, for $\lambda/\mu>0.5$ the Problems~\ref{pr:stepb}--\ref{pr:stepblow} are infeasible. It was shown in~\cite{goseling13peva} that the symmetric random walk with joint departures is ergodic as long as $\lambda/\mu<1$ and $\mu^*>0$. Therefore, non-ergodicity of one of the random walks at hand is not the reason for infeasibility of the linear programs.  
 A more careful examination reveals that in this case the bias terms cannot be bounded by componentwise linear functions. 

In addition to results as a function of $\lambda/\mu$, we provide in Figure~\ref{fig:ncsizeB} the behavior of the bounds as a function of $\mu^*$ for a fixed value of $\lambda/\mu$. Upper bounds $\FF_1$ and $\FF_2$ are given by Problems~\ref{pr:stepb} and~\ref{pr:comparisona} for the case that $\bar\mu_1=\mu-\mu^*$ and $\bar\mu_1=\mu^*$, respectively. Lower bounds $\FF_3$ and $\FF_4$ are given by problem~\ref{pr:stepblow} for the case that $\bar\mu_1=\mu^*$ and $\bar\mu_1=\mu-\mu^*$, respectively. It is clearly reflected in the figure that larger perturbations of the transition rates lead to looser bounds. Note also, that for $\mu^*=1/2$ the original random walk has a product form distribution and the uppper and lower bounds coincide. The results in the figure demonstrate that the comparison result might lead to useful bounds in cases that the error bound result does not. Indeed for $\mu^*/\mu<0.5$, $\FF_1$ does not provide much insight, but $\FF_2$ does. In relation to Figure~\ref{fig:ncsizeB} finally note that a lower bound following from a comparison result does not exist. More precisely, the lower bound equivalent of Problem~\ref{pr:comparisona} is infeasible.

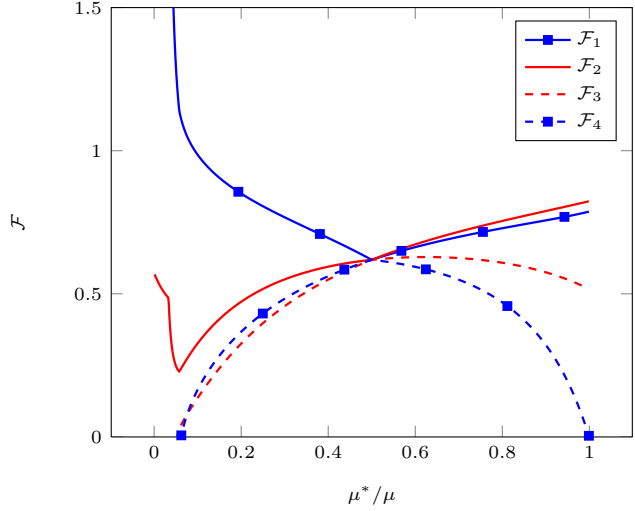
\begin{figure}
\centering
\begin{tikzpicture}
\begin{axis}[
  xlabel=$\mu^*/\mu$,ylabel=$\FF$, 
  ymin=0,ymax=1.5,
  font=\scriptsize,
  legend style={
        cells={anchor=west},
        legend pos=north east,
       font=\scriptsize,
    }
]
\addplot[
  line width=.3mm,color=blue,
mark=square*,mark repeat=150,mark phase=120,mark size=.5mm,mark options={solid}
  ]
table[
  header=false,x index=0,y index=2,
  ]
{matlab_figures/nc_sweep.csv};
\addlegendentry{$\FF_1$};
\addplot[
  mark=none,line width=.3mm,color=red,
  mark=none
  ]
table[
  header=false,x index=0,y index=8,
  ]
{matlab_figures/nc_sweep.csv};
\addlegendentry{$\FF_2$};
\addplot[
  line width=.3mm,color=red,dashed,
  ]
table[
  header=false,x index=0,y index=3,
  ]
{matlab_figures/nc_sweep.csv};
\addlegendentry{$\FF_3$};
\addplot[
  line width=.3mm,color=blue, dashed,
  mark=square*,mark repeat=150,mark phase=0,mark size=.5mm,mark options={solid}
  ]
table[
  header=false,x index=0,y index=1,
  ]
{matlab_figures/nc_sweep.csv};
\addlegendentry{$\FF_4$};
\end{axis}
\end{tikzpicture}
\caption{Marginal first moment, \ie $F(n)=n_1$, of the symmetric random walk with joint departures. Upper bounds in solid lines, lower bounds in dashed lines.($\lambda=0.1$)}
\label{fig:ncsizeB}
\end{figure}

\begin{figure}
\centering
\beginpgfgraphicnamed{pgfx}
\begin{tikzpicture}[scale=.7]
\tikzstyle{axes}=[very thin]
\tikzstyle{trans}=[very thick,-latex]
\tikzstyle{intloop}=[->,to path={
    .. controls +(30:3) and +(-30:3) .. (\tikztotarget) \tikztonodes}]
\tikzstyle{hloop}=[-latex,to path={
    .. controls +(-60:1.5) and +(-120:1.5) .. (\tikztotarget) \tikztonodes}]
\tikzstyle{vloop}=[->,to path={
    .. controls +(-150:1.5) and +(-210:1.5) .. (\tikztotarget) \tikztonodes}]
\tikzstyle{oloop}=[->,to path={
    .. controls +(255:1.5) and +(195:1.5) .. (\tikztotarget) \tikztonodes}]
\draw[axes] (0,0)  -- node[at end, below] {$\scriptstyle \rightarrow n(1)$} (6.5,0);
\draw[axes] (0,0) -- node[at end, left] {$\scriptstyle {\uparrow} {n(2)}$} (0,6.5);
\draw[trans] (0,0) to node[below,at end] {\fm{\lambda_1}} +(1,0);
\draw[trans] (0,0) to node[left, at end] {\fm{\lambda_2}} +(0,1);
\draw[trans] (4,0) to node[at end, below] {\fm{\mu_h}} +(-1,0);
\draw[trans] (4,0) to node[at end, below] {\fm{\lambda_1}} +(1,0);
\draw[trans] (4,0) to node[at end, anchor = south]  {\fm{\lambda_2}} +(0,1);
\draw[trans] (0,4) to node[left, at end] {\fm{\mu_v}} +(0,-1);
\draw[trans] (0,4) to node[at end, anchor = west] {\fm{\lambda_1}} +(1,0);
\draw[trans] (0,4) to node[left, at end] {\fm{\lambda_2}} +(0,1);
\draw[trans] (4,4) to node[at end, anchor = west] {\fm{\lambda_1}} +(1,0);
\draw[trans] (4,4) to node[at end, anchor = south] {\fm{\lambda_2}} +(0,1);
\draw[trans] (4,4) to node[at end, anchor = east] {\fm{\mu_1}} +(-1,0);
\draw[trans] (4,4) to node[at end, anchor = north] {\fm{\mu_2}} +(0,-1);
\draw[very thick,-latex] (0,0) to[oloop] node[below] {\fm{\mu_1+\mu_2}} (0,0);
\draw[very thick,-latex] (4,0) to[hloop] node[below=.4mm] {\fm{\mu_1+\mu_2-\mu_h}} (4,0);
\draw[very thick,-latex] (0,4) to[vloop] node[left] {\fm{\mu_1+\mu_2-\mu_v}} (0,4);
\end{tikzpicture}
\endpgfgraphicnamed
\caption{Random walk with coupled processors.}
\label{fig:coupled}
\end{figure}
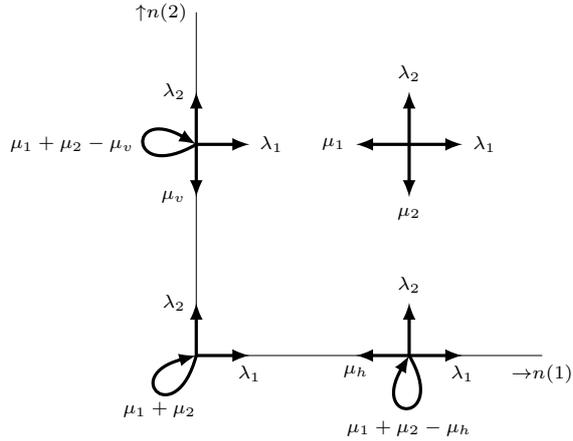

\begin{figure}
\centering
\begin{tikzpicture}
\begin{semilogyaxis}[
  xlabel=$\mu/\rho$,ylabel=$\FF$, 
  ymin=0, 
  font=\scriptsize,
  legend style={
        cells={anchor=west},
        legend pos=north west,
       font=\scriptsize,
    }
]
\addplot[
  mark=none,line width=.3mm,color=blue,
  mark=square*,mark repeat=40,mark phase=25,mark size=.5mm,mark options={solid}
  ]
table[
  header=false,x index=0,y index=8,
  ]
{matlab_figures/coupled_size.csv};
\addlegendentry{$\FF_1$};
\addplot[
  line width=.3mm,color=red,
  ]
table[
  header=false,x index=0,y index=2,
  ]
{matlab_figures/coupled_size.csv};
\addlegendentry{$\FF_2$};
\addplot[
  line width=.3mm,color=red, dashed,
  mark=none
  ]
table[
  header=false,x index=0,y index=1,
  ]
{matlab_figures/coupled_size.csv};
\addlegendentry{$\FF_3$};
\addplot[
  line width=.3mm,color=blue,dashed,
  mark=square*,mark repeat=40,mark phase=10,mark size=.5mm,mark options={solid}
  ]
table[
  header=false,x index=0,y index=7,
  ]
{matlab_figures/coupled_size.csv};
\addlegendentry{$\FF_4$};
\end{semilogyaxis}
\end{tikzpicture}
\caption{Marginal first moment, \ie $F(n)=n_1$, of the symmetric random walk with coupled processors. Upper bounds in solid lines, lower bounds in dashed lines.}\label{fig:coupledsize}
\end{figure}
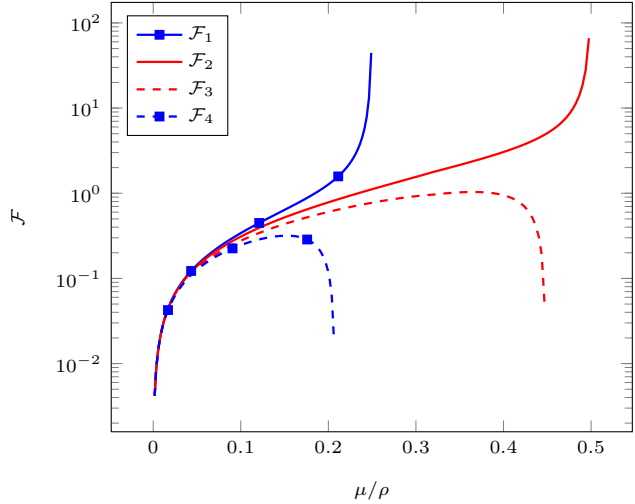

\subsection{Coupled processors}
The next example that we consider is the random walk with coupled processors~\cite{fayolle1979two}. This model arises from a queueing network with two queues, each with a single server. The coupling of the processors is such that in the interior of the state space the processors operate at rates $\mu_1$ and $\mu_2$ respectively. If one of the processors is idle, the other processor adjusts its rates. The transition probabilities are as follows:
\begin{equation*}
\begin{IEEEeqnarraybox}{rCl"rCl"rCl"rCl}
p_{1,e_1}&=&\lambda_1, & p_{1,e_2}&=&\lambda_2, & p_{1,\mea}&=&\mu_h, & p_{1,\o}&=&\mu_1+\mu_2-\mu_h,\\
p_{2,e_1}&=&\lambda_1, & p_{2,e_2}&=&\lambda_2, & p_{2,\meb}&=&\mu_v, & p_{2,\o}&=&\mu_1+\mu_2-\mu_v,\\ 
p_{3,e_1}&=&\lambda_1, & p_{3,e_2}&=&\lambda_2, & p_{3,0}&=&\mu_1+\mu_2,            & \\
p_{4,e_1}&=&\lambda_1, & p_{4,e_2}&=&\lambda_2, & p_{4,\mea}&=&\mu_1,    &  p_{4,\meb}&=&\mu_2,
\end{IEEEeqnarraybox}
\end{equation*}
where $\lambda_1+\lambda_2+\mu_1+\mu_2=1$. The transition diagram is depicted in Figure~\ref{fig:coupled}.

It is known~\cite{fayolle1979two} that this random walk has product-form stationary distribution if and only if $\mu_h+\mu_v=\mu_1+\mu_2$. In that case the parameters $\ra$, $\rb$, of the geometric distribution can be found as the unique solution of $\ra$, $\rb$ in $[0,1]^2$ of the following system of equations:
\begin{align*}
\ra^{-1}\lambda_1 + \ra\mu_h + \rb\mu_2 &= \lambda_1+\lambda_2+\mu_h, \\
\rb^{-1}\lambda_2 + \ra\mu_1 + \rb\mu_v &= \lambda_1+\lambda_2+\mu_h, \\
\ra\mu_h + \rb\mu_v &= \lambda_1+\lambda_2, \\
\ra^{-1}\lambda_1 + \rb^{-1}\lambda_2 + \ra\mu_1 + \rb\mu_2 &= 1,
\end{align*}
that represent the balance equations in each of the components of the state space.

Even though an expression for the generating function of $\pi(n)$ is given in~\cite{fayolle1979two} also for the case that $\mu_h+\mu_v\neq\mu_1+\mu_2$, it is not trivial to use the results from~\cite{fayolle1979two} to evaluate various performance measures. Therefore, the bounds that are given in this paper provide a convenient means of evaluating the performance of a random walk with coupled processors. 

In Figure~\ref{fig:coupledsize} we have presented numerical results for the case that $\lambda_1=\lambda_2=\lambda$, $\mu_1=\mu_2=\mu$, $\mu_h=\mu_v=\mu^*$. The figure presents bounds on the first marginal moment, \ie $F(n)=n_1$, as a function of the system load $\lambda/\mu$ for $\mu^*=0.4\mu$. The perturbed system that we use for all bounds has transition probabilities $\bar\mu_h=\mu^*$ and $\bar\mu_v=2\mu-\mu^*$. The upper bound $\FF_2$ and lower bound $\FF_3$ result from Problems~\ref{pr:stepb} and~\ref{pr:stepblow}, respectively. In addition we have presented upper bound $\FF_1$ and lower bound $\FF_4$ that arise from putting the additional constraints to Problems~\ref{pr:stepb} and~\ref{pr:stepblow}, respectively. These constraints require the functions $A_u$, $B_u$, $u\in\N$, to be linear. Note that this is a stronger constraint than the componentwise linear condition that is imposed in Problems~\ref{pr:stepb} and~\ref{pr:stepblow}. It is clearly reflected in Figure~\ref{fig:coupledsize} that bounding the bias terms with componentwise linear functions significantly improves performance over bounding with (completely) linear functions.

%
%
%
\section{Discussion} \label{sec:disc}
In this paper we have presented a linear programming approach to establishing error bounds for random walks in the quarter-plane. Thereby we obtain the first generic method of establishing such bounds for a large class of processes. The current work can be extended in a multitude of directions, some of which include extensions to higher dimensional random walks and random walks on bounded state spaces. Another extension of interest is to include an optimization over the perturbed system into the optimization problem that used to establish the error bound.

%
%
%
\section*{Acknowledgments}
The authors wish to thank Nico van Dijk for useful discussions. This work is supported by the Netherlands Organization for Scientific Research (NWO), grant $612.001.107$.

%
%
%

\bibliographystyle{amsplain}
\bibliography{IEEEabrv,GoselingBoucherieVanOmmeren14}

\end{document}